\documentclass[11pt,reqno]{amsart}
\usepackage{amssymb,mathrsfs}
\usepackage{ifthen}
\usepackage{colortbl}
\usepackage{cases,appendix}
\definecolor{black}{rgb}{0.0, 0.0, 0.0}
\definecolor{red}{rgb}{1.0, 0.5, 0.5}

\provideboolean{shownotes} 
\setboolean{shownotes}{true} 
\newcommand{\margnote}[1]{
\ifthenelse{\boolean{shownotes}}%
{\marginpar{\raggedright\tiny\texttt{#1}}}%
{}%
}
\newcommand{\hole}[1]{
\ifthenelse{\boolean{shownotes}}%
{\begin{center} \fbox{ \rule {.25cm}{0cm} \rule[-.1cm]{0cm}{.4cm}
\parbox{.85\textwidth}{\begin{center} \texttt{#1}\end{center}} \rule
{.25cm}{0cm}}\end{center}} {} }

\title[Complete synchronization problem for the Kuramoto model with inertia]{Emergent dynamics of the Kuramoto ensemble under the effect of inertia}

\author[Choi]{Young-Pil Choi}
\address[Young-Pil Choi]{\newline Department of Mathematics
    \newline Inha University, Incheon, 402-751, Korea}
\email{ypchoi@inha.ac.kr}

\author[Ha]{Seung-Yeal Ha}
\address[Seung-Yeal Ha]{\newline Department of Mathematical Sciences and Research Institute of Mathematics\newline
    Seoul National University, Seoul 08826, Korea\newline
    Korea Institute for Advanced Study, Hoegiro 85, Seoul, 02455, Korea}
\email{syha@snu.ac.kr}

\author[Morales]{Javier Morales}
\address[Javier Morales]{\newline Department of Mathematics \newline
University of Texas at Austin, Austin, United States}
\email{jmorales@math.utexas.edu}

\subjclass{}

\keywords{}

\numberwithin{equation}{section}

\newtheorem{theorem}{Theorem}[section]
\newtheorem{lemma}{Lemma}[section]
\newtheorem{corollary}{Corollary}[section]
\newtheorem{proposition}{Proposition}[section]
\newtheorem{remark}{Remark}[section]
\newtheorem{definition}{Definition}[section]

\newcommand{\R}{\mathbb R}

\newcommand{\T}{\mathbb T}
\newcommand{\bbr}{\mathbb R}

\newcommand{\Z}{\mathbb Z}

\newcommand{\mc}{\mathcal C}

\newcommand{\bq}{\begin{equation}}
\newcommand{\eq}{\end{equation}}
\newcommand{\e}{\varepsilon}
\newcommand{\lt}{\left}
\newcommand{\rt}{\right}
\newcommand{\lal}{\langle}
\newcommand{\ral}{\rangle}
\newcommand{\pa}{\partial}

\newcommand{\mw}{\mathcal{W}}

\newcommand{\me}{\mathcal{E}}

\newcommand{\pp}{\mathcal{P}}

\newcommand{\comment}[1]{{\color{blue} #1}}

\begin{document}
\allowdisplaybreaks

\date{\today}


\thanks{\textbf{Acknowledgment.} The work of  Y.-P. Choi was supported by NRF grant(NRF-2017R1C1B2012918), the work of J. Morales was supported by the ERC grant Regularity and Stability in Partial Differential Equations(RSPDE), and the work of S.-Y. Ha was  supported by the Samsung Science and Technology Foundation under Project(Number SSTF-BA1401-03)}

\begin{abstract}
We study the emergent collective behaviors for an ensemble of identical Kuramoto oscillators under the effect of inertia. In the absence of inertial effects, it is well known that the generic initial Kuramoto ensemble relaxes to the phase-locked states asymptotically (emergence of complete synchronization) in a large coupling regime. Similarly, even for the presence of inertial effects, similar collective behaviors are observed numerically for generic initial configurations in a large coupling strength regime. However, this phenomenon has not been verified analytically in full generality yet, although there are several partial results in some restricted set of initial configurations. 
In this paper, we present several improved complete synchronization estimates for the Kuramoto ensemble with inertia in two frameworks for a finite system. Our improved frameworks describe the emergence of phase-locked states and its structure. Additionally, we show that as the number of oscillators tends to infinity, the Kuramoto ensemble with infinite size can be approximated by the corresponding
 kinetic mean-field model uniformly in time. Moreover, we also establish the global existence of measure-valued solutions for the Kuramoto equation and its large-time asymptotics. 
\end{abstract}

\maketitle \centerline{\date}

\tableofcontents

%
%
%
%
\section{Introduction}  \label{sec:1}
\setcounter{equation}{0}
Collective behaviors of complex system is one of the important characteristics that are often observed in classical and quantum many-body systems, e.g., synchronous firing of flash, swarming of fish and flocking of birds, array of Josephson junctions, etc \cite{AB, A-D, B-S, BB, H-K-P-Z, P-R-K, Str, Wi1}. Thus, one of the natural questions is what makes collective behavior emerge in complex systems? Good understanding of such emergent collective behaviors will certainly provide a design principle for man-made interacting systems like a robot system and multi-agent system of UAVs. Among such diverse collective behaviors, our main interest in this paper lies on the synchronization phenomenon, roughly speaking ``adjustment of rhythms in weakly coupled oscillators due to weak interactions", i.e., oscillators adjust their rhythms and exhibit a common rhythm like a single giant oscillators. Despite of ubiquitous occurrence in our nature, its systematic study based on the mathematical model was only done about a half century ago by Winfree \cite{Wi2} and Kuramoto \cite{Ku}. In the literature, the Kuramoto model has been extensively studied in the physics and the control theory communities \cite{AB, P-R-K}.  The Kuramoto model can be formally derived from the linearly coupled space-homogeneous Ginzburg-Landau equations. For more detailed explanation, we refer review papers and books \cite{Ku, P-R-K}.  In this paper, we focus on a generalized Kuramoto model with inertia which was introduced in \cite{Erm} to explain the slow relaxation of firefly Pteroptyx malaccae's rhythmn (see \cite{TLO1, TLO2} for physical differences with the Kurmaoto model from the view point of phase-transitions at the critical coupling strengths). To describe our governing system, consider a steady, undirected and connected network ${\mathcal G}= ({\mathcal V}, {\mathcal E}, {\mathcal A})$ made of a vertex-set ${\mathcal V} := \{ 1, 2,\dots, N \}$, an edge set ${\mathcal E}$ and a capacity matrix ${\mathcal A} = (a_{ij})$. Assume that inertial Kuramoto oscillators are located on the vertices of the network. Let $\theta_i = \theta_i(t)$ be the phase of the Kuramoto oscillator at vertex $i$, and we denote $m_i$ by the strength of inertia (mass) of the $i$-th oscillator. In this setting, the temporal dynamics of the phase $\theta_i$ is given by the following second-order system: 
\begin{equation} 
\begin{cases} \label{main}
\displaystyle m_i\ddot{\theta}_i = -\gamma_i \dot{\theta}_i + \nu_i +  \kappa \sum_{j=1}^N a_{ij} \sin(\theta_j - \theta_i),\quad  t> 0, \quad  i =1,\cdots,N,  \\
\displaystyle (\theta_i,\dot{\theta}_i)(0) = (\theta^{in}_{i}, \omega^{in}_{i}),
\end{cases}
\end{equation}
where $\kappa$ denotes the uniform coupling strength between oscillators, and $\nu_i$ is the natural frequency of the $i$-th oscillator. Such a natural frequency is assumed to be a steady random variable with $g = g(\nu)$ as its probability density function. We also assume that the capacity matrix ${\mathcal A}$ has nonnegative components and is symmetric:
\[ a_{ij} \geq 0, \quad a_{ij} = a_{ji} \quad \mbox{for all $i, j = 1, \cdots, N.$} \]
Note that forcing terms in the R.H.S. of $\eqref{main}_1$ represent the linear friction, intrinsic randomness, and the nonlinear couplings between oscillators, respectively. Thus, due to the dynamic interplay between these three forcing mechanisms, many interesting features can emerge in their dynamics \cite{CCHKK, C-H-J-S, HL16}.

In this paper, we are mainly interested in the ``{\it complete synchronization problem} (in short CSP)" which can be stated as follows. 
\begin{quote}
Find a sufficient framework for the complete synchronization of system \eqref{main} in terms of system parameters such as inertia, friction and coupling strengths. More precisely, find a coupling strength $\kappa_c$ such that for $\kappa > \kappa_c$, any Kuramoto phase vector $\Theta(t) = (\theta_1(t), \cdots, \theta_N(t))$ with generic initial data $\Theta^{in}$, the complete synchronization occurs in the following sense:
\[ 
\lim_{t \to \infty} |{\dot \theta}_i(t) - {\dot \theta}_j(t)| = 0, \quad 1 \leq i, j \leq N. 
\]
\end{quote}
i.e., all oscillators have the same frequency asymptotically.  For the zero inertia case, i.e., $m_i = 0$ for $i=1,\cdots, N$, the CSP has been extensively investigated in \cite{BCM, C-H-J-K, D-B, V-M1, V-M2} and has been resolved in a recent work \cite{H-K-R}. Of course, there are many interesting open problems. However, in the presence of inertia $m_i > 0$, as far as the authors know, the CSP is still an interesting open problem in the nonlinear dynamics of Kuramoto oscillators. 

Next, we briefly discuss our three main results. First, we present an improved complete synchronization estimate for the ensemble of identical Kuramoto oscillators with homogeneous inertia ($m_i = m$) and all-to-all coupling $(a_{ij} = N^{-1})$. Previously, an admissible initial phase distribution which results in the complete synchronization was when all oscillators were confined at most in a half circle (see Theorem \ref{T3.1}). In this paper, we show that if the initial configuration $\{ (\theta_i^{in}, \omega_i^{in})\}$ and the coupling strength $\kappa$ satisfy 
\[ 
R_p(\Theta^{in}) := \lt| \frac{1}{N} \sum_{j=1}^{N} e^{{\mathrm i} \theta_j^{in}} \rt| > 0 \quad \mbox{and} \quad \kappa >  \kappa_* := \frac{m}{N} \frac{\sum_{i=1}^N |\omega^{in}_i|^2}{R_p(\Theta^{in})^2}, 
\]
then the Kuramoto phase vector $\Theta(t)$ starting from $\Theta^{in}$ tends to a phase-locked state asymptotically and the resulting phase-locked state is either one-phase cluster or bi-polar state (see Theorem \ref{T3.4}). Second, we consider the CSP in a perturbative setting of the first result: there exist positive constants ${\bar a}$ and $E_*$ such that if
\[ 
\gamma_i > 0, \quad m_i > 0,  \quad |a_{ij} - {\bar a} | \ll 1, \quad  \sum_{i,j=1}^Na_{ij}\cos(\theta^{in}_i - \theta^{in}_j) - \sum_{i=1}^N m_i |\omega^{in}_i|^2 > E_*. 
\]
then the CSP occurs asymptotically and the Kuramoto phase vector tends to a phase-locked states which is either one-phase cluster or bi-polar state (see Theorem \ref{T3.5}). Our last main result deals with system \eqref{main} when $N \gg 1$. In this case, we show that system \eqref{main} for identical oscillators with homogeneous inertia and friction can be well approximated by the corresponding mean-field kinetic equation uniformly in time. To achieve this, we show a uniform stability estimate in the Wasserstein-2 distance, and we also prove that the corresponding mean-field model has a global measure-valued solution with an emergent asymptotic property (see Theorem \ref{T3.6}). 

The rest of the paper is organized as follows. In Section \ref{sec:2}, we review the elementary theoretical minimum on the relative equilibria, namely phase-locked states, a gradient flow-like formulation, order parameters, and discuss the basic properties of the kinetic Kuramoto model. In Section \ref{sec:3}, we discuss previous results on the CSP and summarize our three main results. In Section \ref{sec:4}, we present the CSP for the ensemble of identical Kuramoto oscillators with homogeneous inertia and discuss a sufficient framework leading to the CSP. In Section \ref{sec:5}, we present the CSP for the ensemble of identical Kuramoto oscillators with heterogeneous inertia and present a sufficient framework for the CSP.  In Section \ref{sec:6}, we discuss similar results for the kinetic Kuramoto equation and provide a uniform-in-time mean-field limit, the global existence of measure-valued solution and an emergence asymptotic behavior. Finally, Section \ref{sec:7} is devoted to a brief summary of our main results.

\bigskip

\noindent {\bf Notation}: We let $\Theta = (\theta_1, \cdots, \theta_N)$ and $\Omega = (\omega_1, \cdots, \omega_N)$ be the phase and frequency vectors, respectively. Additionally, for a vector $Z= (z_1, \cdots, z_N) \in \bbr^N$, we let $\ell^2$-norm be defined as follows:
\[ 
\|Z\| := \lt( \sum_{i=1}^{N} |z_i|^2 \rt)^{\frac{1}{2}}. 
\]

%
%
%
%
\section{Preliminaries} \label{sec:2}
\setcounter{equation}{0}
In this section, we review the theoretical minimum for our discussions in the subsequent sections.

 \subsection{Relative equilibria} \label{sec:2.1}
 In this subsection, we discuss the concept of equilibria for the Kuramoto model. To discuss the equilibrium of \eqref{main}, we rewrite  the second-order system \eqref{main} as a first-order system on phase-frequency space:
\begin{equation}
\begin{cases} \label{main-1}
\displaystyle {\dot \theta}_i = \omega_i, \quad t > 0,~~i = 1, \cdots, N, \\[2mm]
\displaystyle {\dot \omega}_i = \frac{1}{m_i} \lt( -\gamma_i \omega_i + \nu_i +  \kappa \sum_{j=1}^N a_{ij} \sin(\theta_j - \theta_i)  \rt). 
\end{cases}
\end{equation}
Since the R.H.S. of \eqref{main-1} is $2\pi$-periodic in $\theta$-variable, the natural state space of configurations associated with \eqref{main-1} is given by $\T^N \times \bbr^N$. However, we can embed system \eqref{main-1} as a dynamical system in $\bbr^N \times \bbr^N$. Throughout the paper, we will regard system \eqref{main-1} as a dynamical system on $\R^N \times \bbr^N$. Thus, the boundedness of trajectory is a priori not a trivial matter.  

Note that the equilibrium solution $(\theta_i, \omega_i)$ to \eqref{main-1} corresponds to the solution of the following system:
\begin{equation*} 
\omega_i = 0, \quad -\gamma_i \omega_i + \nu_i+ \kappa \sum_{j=1}^N a_{ij} \sin(\theta_j - \theta_i)  = 0 \quad \mbox{for} \quad 1\leq i \leq N,
\end{equation*}
i.e., the equilibrium phase $\{\theta_i \}$ should satisfy the following equations:
\begin{equation} \label{A-2}
 \nu_i + \kappa \sum_{j=1}^N a_{ij} \sin(\theta_j - \theta_i)  = 0, \quad 1 \leq i \leq N,
\end{equation}
which is the same equilibrium system for the Kuramoto model. 

By the anti-symmetry of $a_{ij} \sin(\theta_j - \theta_i)$ in the exchange $(i,j)~\leftrightarrow~(j,i)$, we can see that if the equilibrium system \eqref{A-2} is solvable, then $\sum_{i=1}^{N} \nu_i  = 0$. Thus, if $\sum_{i=1}^{N} \nu_i  \not = 0$, then there are no equilibrium solution to \eqref{main-1}. This is exactly the same as the situation for $N$-body system in Newtonian mechanics. In order to continue the discussion on the equilibria, we need to consider the motion of relative equilibrium, which is equilibrium in a rotating frame with average natural frequency $\nu_c := \frac{1}{N} \sum_{i=1}^{N} \nu_i$.  In this case, the relative equilibria becomes a train-like ensemble rotating on the unit circle. Next, we recall some terminologies to be used throughout the paper.
\begin{definition} \label{D2.1}
1. A phase-locked state $\Theta = (\theta_1, \cdots, \theta_N)$ of \eqref{main} is defined to be a solution to \eqref{main} if and only if $\theta_i$ satisfies
\[  
|\theta_i(t) - \theta_j(t)| =  |\theta_i^{in} - \theta_j^{in}|, \quad 1 \leq i, j \leq N, \quad t \geq 0. 
\]
2. A complete synchronization occurs asymptotically if and only if  relative frequency differences tend to zero asymptotically:
\[ \lim_{t \to \infty} \max_{1 \leq i,j \leq N}  |{\dot \theta}_i(t) - {\dot \theta}_j(t)| = 0. \]
\end{definition} 
\begin{remark} \label{R1.1} Note that the original Kuramoto model corresponds to the complete graph with the capacity $a_{ij} = 1/N$, $m_i=0$, $\gamma_i=1$ for $i=1,\cdots, N$. 
Many previous works \cite{AB, C-H-J-K, C-Sp, D-X, D-B, M-S1, M-S2, M-S3} have conducted  rigorous analysis of the complete or partial synchronizations for the Kuramoto model, but there have been few studies \cite{V-M1,V-M2} on the existence of phase-locked states.
\end{remark}
\subsection{A gradient flow-like formulation}\label{sec:2.2} We introduce weighted phase and frequency averages as follows.
\[
\theta_s := \frac1N\sum_{i=1}^N \gamma_i \theta_i, \quad \omega_s := \frac1N\sum_{i=1}^N m_i \omega_i, \quad \mbox{and} \quad \nu_c := \frac{1}{N} \sum_{i=1}^{N} \nu_i.
\]
Then we can easily find from \eqref{main} that
\[
\dot \omega_s + \dot \theta_s = \nu_c,
\]
due to the symmetry of the capacity matrix $\mathcal{A}$. On the other hand, if we consider the Kuramoto ensemble with uniform inertia and friction:
\[ 
m_i = m, \quad \mbox{and} \quad \gamma_i = \gamma, \quad 1 \leq i  \leq N,
\]
then system \eqref{main} becomes 
\begin{equation} \label{B-0}
m \ddot{\theta}_i = -\gamma \dot{\theta}_i  + \nu_i + \kappa \sum_{j=1}^N a_{ij} \sin(\theta_j - \theta_i),\quad  i =1,\cdots,N, \quad t > 0.
\end{equation}
We introduce the phase and frequency averages:
\[ 
\theta_c := \frac{1}{N} \sum_{i=1}^N \theta_i \quad \mbox{and} \quad \omega_c := \frac{1}{N} \sum_{i=1}^{N} \omega_i
\]
\begin{lemma} \label{L2.1}
Let $\theta_i = \theta_i(t)$ be a solution to \eqref{B-0}. Then, the total phase and frequency are explicitly given by the following relations: 
\begin{align}
\begin{aligned} \label{B-1}
& (i)\,\,\,\,\theta_c(t) = \theta_c(0) + t \nu_c + m(\omega_c(0) - \nu_c)(1 - e^{-\frac{\gamma t}{m}}),   \cr
& (ii)\,\,\omega_c(t) = \nu + (\omega_c(0) - \nu_c)e^{-\frac{\gamma t}{m}},
\end{aligned}
\end{align}
for $t \geq 0$. 
\end{lemma}
\begin{proof} (i) We sum \eqref{B-0} over all $i$ and divide the resulting relation by $N$ to find 
\[ \label{B-0-1}
m {\ddot \theta}_c + \gamma {\dot \theta}_c = \nu_c, \quad t > 0.
\]
This yields the first desired estimate. \newline

\noindent (ii)  We differentiate the formula for $\theta_c$ with respect to $t$ to get the desired estimate for average frequency. 
\end{proof}

\begin{remark} \label{R2.2} It follows from the explicit formulas in \eqref{B-1} that 
\[
\lim_{t \to \infty}\Big |\theta_c(t) - \Big(\theta_c(0) + m(\omega_c(0) - \nu_c) + t\nu_c \Big) \Big| = 0 \quad \mbox{and} \quad \lim_{t \to \infty} \omega_c(t) = \nu_c.
\]
Thus, the resulting phase-locked states through the Kuramoto flow \eqref{main} with inertia have the same structure as the Kuramoto model with zero inertia $m= 0$, i.e., inertia does not play any role in the structure of phase-locked states. 
\end{remark}

Next, we present a gradient flow-like formulation \cite{V-W}. For this, we set 
\begin{equation} \label{B-2}
V(\Theta) := - \sum_{j=1}^{N} \nu_i \theta_j +\frac{\kappa}{2} \sum_{i,j=1}^{N} a_{ij} (1-\cos(\theta_{i}-\theta_{j}) ).
\end{equation}
Then, it is easy to see that \eqref{main} can be rewritten as a second-order gradient-like system \cite{CL, CHLXY}:
\begin{equation} \label{B-3}
M{\ddot \Theta} + \Gamma {\dot \Theta}  =  -\nabla V(\Theta),  \quad t > 0, 
\end{equation}
where $M:=diag\{m_1,  \dots, m_N\}$, $\Gamma:=diag\{\gamma_1, \dots, \gamma_N\}$, and $\nabla V$ is the gradient of $V$ with respect to $\Theta$, i.e.,
\begin{align*}
\begin{aligned}
\nabla V(\Theta) &=-(\nu_1, \cdots, \nu_N) - \kappa \Gamma(\theta) \\
&:= -(\nu_1, \cdots, \nu_N) - \kappa \lt(\sum_{j=1}^N a_{1j} \sin(\theta_j - \theta_1), \dots, \sum_{j=1}^{N} a_{Nj} \sin(\theta_j - \theta_N) \rt).
\end{aligned}
\end{align*}
Note that the set of all equilibria coincides with the set of critical points of the gradient vector field $\nabla V$:
\[
\mathcal{S}:=\lt\{ \Theta \in \R^N : \nabla V(\Theta) = 0\rt\}.
\]
Next, we recall a convergence result for the gradient-like system \eqref{B-2} -\eqref{B-3} on $\bbr^{2N}$ without a proof. 
\begin{proposition} \label{P2.1}
\emph{\cite{CHLXY,HJ}}
Let $\Theta$ be a solution to \eqref{B-3} with a finite $W^{1, 2}(\R_+)$-norm: 
\[ 
\|\Theta\|_{W^{1, 2}} :=  \|\Theta\| + \|{\dot \Theta}\| < \infty.
\]
Then, there exists  $\Theta^\infty \in {\mathcal S}$ such that
\[\lim_{t\to \infty} \Big( \|{\dot \Theta}(t) \| +\|\Theta(t)-\Theta^\infty\| \Big) =0.\]
\end{proposition}
\begin{remark}\label{R2.1} 
Note that $\omega_i$ satisfies
\[ 
m_i \dot{\omega}_i = -\gamma_i \omega_i  + \nu_i + \kappa \sum_{j=1}^N a_{ij} \sin(\theta_j - \theta_i). 
\]
This yields
\[ 
\frac{d|\omega_i|}{dt} \leq -\frac{\gamma_i}{m_i} |\omega_i| + \frac{1}{m_i} \lt(|\nu_i| + \kappa \sum_{j=1}^N a_{ij} \rt)
\]
for $m_i > 0$. Then, Gronwall's lemma yields
\[ 
|\omega_i(t) | \leq |\omega_i^{in}|e^{-\frac{\gamma_i}{m_i}t} + \frac{1}{\gamma_i} \lt(|\nu_i| + \kappa \sum_{j=1}^N a_{ij} \rt)\lt(1  -  e^{-\frac{\gamma_i}{m_i}t} \rt)
\]
for $\gamma_i > 0$.
Thus, we obtain the uniform boundedness of $\omega_i$:
\[   |\omega_i(t) | \leq \max \lt\{   \gamma_i^{-1}\lt(|\nu_i| + \kappa \sum_{j=1}^N a_{ij}\rt) ,   |\omega_i^{in}| \rt\}. \]
Then, it follows from Proposition \ref{P2.1} that it suffices to show uniform-in-time boundedness of $\Theta$ to guarantee the formation of asymptotic phase-locked states. In fact, in the absence of inertia, the uniform boundedness of phase vector are verified in \cite{H-K-R}, which leads the resolution of the complete synchronization problem for the Kuramoto model without inertia. 
\end{remark}

\subsection{Order parameters} \label{sec:2.3} In this subsection, we recall real-valued order parameters measuring the degree of synchronization of the Kuramoto phase ensemble. For a given phase vector $\Theta = (\theta_1, \cdots, \theta_N) \in \bbr^N$ and for every $i \in \{1,...,N\}$, we introduce the {\it local order parameter} $R_i$ and the {\it average phase} $\varphi_i$ via the following relation:
\bq\label{F-3}
R_{p,i} (t)e^{{\mathrm i} \varphi_{p,i}(t)} := \sum_{j=1}^N a_{ij} e^{{\mathrm i} \theta_j(t)}.
\eq
Here the subscript ``$p$" in $R_p$ and $\varphi_p$ was introduced to distinguish the order parameters for the corresponding kinetic model in Section \ref{sec:6}. Note that
\[
0 \leq R_{p,i}(t) \leq \max_{1 \leq j \leq N} a_{ij} \quad \mbox{for} \quad t \geq 0.
\]
We divide \eqref{F-3} by $e^{{\mathrm i} \varphi_{p,i}}$ and compare real and imaginary parts to find 
\begin{equation*} 
R_{p,i} =  \sum_{j=1}^N a_{ij} \cos(\theta_j - \varphi_{p,i}) \quad \mbox{and} \quad 0 = \sum_{j=1}^N a_{ij} \sin(\theta_j - \varphi_{p,i}).
\end{equation*}
Similarly, we have
\begin{align}
\begin{aligned} \label{F-3-2}
R_{p,i}  \cos(\varphi_{p,i} - \theta_i) &=  \sum_{j=1}^N a_{ij} \cos(\theta_j - \theta_i), \\
R_{p,i} \sin(\varphi_{p,i} - \theta_i)  &= \sum_{j=1}^N a_{ij} \sin(\theta_j - \theta_i).
\end{aligned}
\end{align}
Then \eqref{main} can be rewritten as
\[
m_i\ddot{\theta}_i = -\gamma_i\dot{\theta}_i - \kappa R_{p,i} \sin(\theta_i - \varphi_{p,i}),\quad  i =1,\cdots,N, \quad t > 0.
\]
In the homogeneous and all-to-all coupling case, i.e., $m_i = m$, $\gamma_i = \gamma$, and $a_{ij} = 1/N$, the order parameters $R_p$ and $\varphi_p$ are given by the following relation:
\begin{equation}\label{B-4}
 R_pe^{{\mathrm i} \varphi_p} := \frac1N \sum_{j=1}^N e^{{\mathrm i} \theta_j}.
\end{equation}
Then for the same reason, we see that 
\begin{equation} \label{B-4-1}
\sum_{j=1}^N \cos(\theta_j - \varphi_p) = NR_p \quad \mbox{and} \quad \sum_{j=1}^N \sin(\theta_j - \varphi_p) = 0.
\end{equation}
Next, we use the identities \eqref{B-4-1} to obtain
\begin{align}
\begin{aligned} \label{B-4-2}
\sum_{i,j=1}^N \cos(\theta_i - \theta_j) &=  \sum_{i,j=1}^N \cos(\theta_i - \varphi_p + \varphi_p - \theta_j)  \\
&=   \sum_{i,j=1}^N \lt( \cos(\theta_i - \varphi_p) \cos(\theta_j -\varphi_p) + \sin(\theta_i - \varphi_p) \sin(\theta_j-\varphi_p) \rt)   \\
&=  (NR_p)^2.
\end{aligned}
\end{align}
Thus,  thanks to \eqref{B-4-2}, the potential function $V(\Theta)$ in \eqref{B-2} can also be rewritten in terms of order parameter $R_p$ as follows.
\begin{equation*} 
V(\Theta)=-\sum_{j=1}^{N} \nu_j \theta_j -\frac{\kappa N}{2} \lt(1 -R_p^2\rt). 
\end{equation*}
On the other hand, we divide \eqref{B-4} by $e^{{\mathrm i} \theta_i}$ and compare the imaginary part of the resulting relation to find 
\begin{equation} \label{B-5}
 R_p \sin (\varphi_p - \theta_i ):= \frac1N \sum_{j=1}^N \sin(\theta_j  - \theta_i). 
\end{equation}
Then we can use \eqref{B-5} to rewrite \eqref{main} with homogeneous inertia and friction and all-to-all coupling in the alternative form:
\begin{equation} \label{B-6}
m\ddot{\theta}_i + \dot{\theta}_i = \nu - \kappa R_p \sin(\theta_i - \varphi_p), \quad i=1,\cdots,N.
\end{equation}

\subsection{A kinetic Kuramoto model with inertia} \label{sec:2.3}
Let $f = f(\theta,\omega,t)$ be the one-oscillator distribution function in $\T$ with the frequency $\omega$ at time $t$. Then, as the number of oscillators $N \to \infty$, the following Vlasov-type kinetic
equation for $f$ can be formally derived from the particle system \eqref{F-1} using the BBGKY hierarchy: 
\begin{equation} 
\begin{cases} \label{B-7}
\displaystyle  \partial_t f + \partial_\theta(\omega f) +\partial_\omega (\mathcal{F}_a[f]f) = 0, \quad (\theta,\omega) \in \T \times \R , \quad t > 0, \\[2mm]
\displaystyle \mathcal{F}_a[f](\theta,\omega,t) = \frac1m\lt( -\gamma \omega + \kappa \int_{\T \times \R}\sin(\theta_* - \theta) f(\theta_*, \omega_*,t) \,d \theta_* d \omega_*\rt). \\
\end{cases}
\end{equation}
The rigorous derivation of the kinetic equation \eqref{B-7} via the mean-field limit 
is discussed in \cite{CHY2} as well as the global existence of measure-valued solutions. More precisely, let $\pp_2(\T \times \R)$ be the set of all Radon measures on $\T \times \R$ with finite second moment and unit mass and let $\mc_w([0,T); \pp_2(\T \times \R))$ be the space of all weakly continuous time-dependent measures. If the initial measure $f^{in} \in \pp_2(\T \times \R)$ is compactly supported in the frequency variable, then there exists $f$ in $\mc_w([0,T); \pp_2(\T \times \R))$ satisfying  equation \eqref{B-7} in the sense of distributions. Note that this way of constructing the measure-valued solutions provides the rigorous mean-field limits from the particle system \eqref{main-1} to the kinetic equation \eqref{B-7}. For any $T >0$, this mean-field limit holds until such a time.
However, the time $T$ cannot be infinity since the upper bound of the error estimate between the empirical measure $\displaystyle f^N(\theta,\omega,t) := \frac1N\sum_{i=1}^N \delta_{(\theta_i(t),\omega_i(t))}$ and the solution $f \in \mc_w([0,T); \pp_2(\T \times \R))$ depends on $T$.  Recently, in \cite{H-K-Z}, the uniform-in-time mean-field limit for a flocking model is established based on the uniform-in-time stability of solutions. Inspired by that work, we  provide the uniform-in-time mean-field limit of the particle system \eqref{main-1} to derive the kinetic equation \eqref{B-7} rigorously for all time $t \geq 0$. For this, we directly follow the strategies used in \cite{CHN, CHLXY}, in which the decay estimate of $\|\Theta\| + \|\dot\Theta\|$ for the system \eqref{main-1} 
on networks is obtained under suitable assumptions on the initial configurations 
(see \cite{CL} for more general settings). Note that a  nonlinear stability estimate in $\ell^\infty$ for the system \eqref{main-1} is obtained in \cite[Theorem 3.1]{CHN}. However, this only gives the stability in $\theta$, not $\omega$.
Additionally, this stability estimate requires strong assumptions on the initial configurations. For this reason, we need to combine the strategies used in \cite{CHN, CHLXY}. \newline

\section{Presentation of main results for CSP} \label{sec:3}
\setcounter{equation}{0}
In this section, we first review the state-of-the-art results on the CSP for \eqref{main} and then we present our three main results. 

 \subsection{The-state-of-the-art results}  As briefly discussed in Introduction, the CSP for the Kuramoto model with the homogeneous inertia is still open. So far, there are two main conceivable approaches for the asymptotic phase-locking of \eqref{main}. The first approach \cite{CHY1} is based on the Lyapunov functional approach employing the phase and frequency diameters as Lyapunov functions. Second approach \cite{CHLXY, CL, LXY} is to use the gradient flow-like formulation discussed in Section \ref{sec:2.1} combined with an energy method. At present, both approaches require a restricted class of initial data which are far from the generic initial configurations. In the sequel, we briefly review the-state-of-the-art results. For the corresponding results for the Kuramoto model with zero inertia, we refer to recent survey papers \cite{D-B2, H-K-P-Z}. As far as the authors know, the CSP for the Kuramoto model with inertia has been first discussed in \cite{CHY1}.  In this sequel, we consider the homogeneous inertia and frictions: 
\[  m_i = m, \quad \gamma_i = 1, \quad a_{ij} = \frac{1}{N}, \quad 1 \leq i, j \leq N. \]
In this case, we consider the Kuramoto model with inertia:
\begin{equation} \label{C-1}
 m \ddot{\theta}_i = -\dot{\theta}_i + \nu_i +   \frac{\kappa}{N} \sum_{j=1}^N \sin(\theta_j - \theta_i),\quad  i =1,\cdots,N, \quad t > 0. 
\end{equation}
 \subsubsection{Identical Kuramoto oscillators} 
We first recall a framework for identical oscillators described in \cite{CHY1}. For $\ell = 1, 2,$ we set 
\[ {\mathcal C}_\ell(0) := \max \{ D(\Theta(0)),~D(\Theta(t), D(\Theta(0)) + m \ell {\dot D}(\Theta(0)) \}. \]
\begin{theorem}  \label{T3.1}
\emph{\cite{CHY1}}
Suppose that parameters $m_i, \gamma_i, \nu_i, \kappa$ and initial data satisfy the following relations:
\begin{eqnarray*}
&& m_i = m, \quad \gamma_i = 1, \quad \nu_i = \nu, \quad 1 \leq i \leq N, \cr
&&  0<  {\mathcal C}_1(0) < \pi, \quad \mbox{and} \quad m \kappa \in \lt(0, \frac{1}{4} \rt) \cup \lt( \frac{{\mathcal C}_1(0)}{4 \sin {\mathcal C}_1(0)}, \infty \rt).  
\end{eqnarray*}
Then, for any solution $(\Theta, {\dot \Theta})$ to \eqref{C-1}, there exist a positive constant $\Lambda_1$ such that 
\[ D(\Theta(t)) + D({\dot \Theta}(t))  \leq {\mathcal O}(1) e^{-\Lambda_1 t}, \quad \mbox{as $t \to \infty$}. \]
\end{theorem}

\subsubsection{Nonidentical Kuramoto oscillators} In this part, we briefly discuss two frameworks (small inertia and large inertia) for CSP in the ensemble of nonidentical Kuramoto oscillators. 

We first present CSP for a small inertia case. For this, we set $\displaystyle D^{\infty}_1 \in \Big(0, \frac{\pi}{2} \Big]$ to be a solution to the trigonometric equation:
\[ D(\nu) :=\max_{i,j} |\nu_i - \nu_j|, \quad  \sin x = \frac{D(\nu)}{\kappa}. \]
\begin{theorem}  \label{T3.2}
\emph{\cite{CHY1}}
Suppose that the parameters $m_i, \gamma_i, \nu_i, \kappa$ and the initial data satisfy the following relations.
\begin{eqnarray*}
&& m_i = m, \quad \gamma_i = 1, \quad 1 \leq i \leq N, \quad 0< D(\nu) < \kappa, \\[2mm]
&&0 < m\kappa < \frac{D^{\infty}_1}{4 \sin D^{\infty}_1}, \quad \mbox{and} \quad 0< {\mathcal C}_2(0) < D^{\infty}_1.
\end{eqnarray*}
Then, for any solution $(\Theta, {\dot \Theta})$ to \eqref{C-1}, there exists a positive constant $\Lambda_2$ such that 
\[ \sup_{0 \leq t < \infty} D(\Theta(t)) < D^{\infty}_1, \quad  D({\dot \Theta}(t))  \leq {\mathcal O}(1) e^{-\Lambda_2 t}, \quad \mbox{as $t \to \infty$}. \]
\end{theorem}
Next, we deal with the CSP with a large inertia.

\begin{theorem}  \label{T3.3}
\emph{\cite{CHY1}}
Suppose that parameters $m_i, \gamma_i, \nu_i, \kappa$ and initial data satisfy the following relations.
\begin{eqnarray*}
&& m_i = m, \quad \gamma_i = 1, \quad 1 \leq i \leq N, \quad 0 < D(\nu) < \frac{\pi}{8m}, \cr
&&  m \kappa \geq \frac{\pi}{8}, \quad 0 < {\mathcal C}_2(0) < 4 m D(\nu). 
\end{eqnarray*}
Then, for any solution $(\Theta, {\dot \Theta})$ to \eqref{C-1}, there exists a positive constant $\Lambda_3$ such that 
\[ \sup_{0 \leq t  < \infty}D(\Theta(t)) \leq 4m D(\nu), \quad D({\dot \Theta}(t)) \leq {\mathcal O}(1) e^{-\Lambda_3 t}, \quad \mbox{as $t \to \infty$}.   \]
\end{theorem}

\subsection{Three main results} \label{sec:3.2} In this subsection, we briefly list our main results for the particle and kinetic Kuramoto models with inertia, and compare them with previous results discussed in previous subsection. 
\subsubsection{The Kuramoto model}  In this part, we summarize two main results on the CSP for the particle Kuramoto model under two frameworks. \newline

\noindent $\bullet$ {\bf Setting A}:  Assume that 
\[  m_i = m, \quad \gamma_i = \gamma, \quad \nu_i = 0, \quad a_{ij} = \frac{1}{N}, \quad 1 \leq i, j \leq N. \]
Under this setting, the Kuramoto model \eqref{main} becomes 
\begin{equation} \label{C-10}
m \ddot{\theta}_i =-\gamma \dot{\theta}_i + \frac{\kappa}{N}\sum_{j=1}^N \sin(\theta_j - \theta_i),\quad  i =1,\cdots,N, \quad t > 0.
\end{equation}
This can be expressed as an alternative form of \eqref{D-1} for $({\theta}_i, \omega_i)$:
\begin{equation} \label{C-11}
\begin{cases}
\displaystyle {\dot \theta}_i = \omega_i, \quad t > 0, \\
\displaystyle {\dot \omega}_i = -\frac{1}{m} \lt( \gamma \omega_i - \frac{\kappa}{N}\sum_{j=1}^N \sin(\theta_j - \theta_i) \rt).
\end{cases}
\end{equation}
Note that the equilibrium solution $(\theta_i^{\infty}, \omega_i^{\infty})$ for \eqref{C-11} is a solution of the following system:
\begin{equation} \label{C-12}
\omega_i^{\infty}  = 0, \quad \sum_{j=1}^N \sin(\theta^{\infty}_j - \theta^{\infty}_i) = 0, \quad 1 \leq i \leq N,
\end{equation}
which coincides with an equilibrium system of the Kuramoto model without inertia:
\[ {\dot \theta}_i = \frac{\kappa}{N} \sum_{j=1}^{N} \sin (\theta_j - \theta_i), \quad i = 1, \cdots, N. \]
Thus, it follows from \eqref{B-5} and \eqref{C-12} that equilibrium solution $(\theta_i^\infty, \omega_i^\infty)$ satisfies
\begin{equation} \label{C-13}
 R_p^{\infty} \sin(\theta^\infty - \theta_i^{\infty}) = 0, \quad 1 \leq i \leq N,
\end{equation} 
where $R_p^{\infty} := R_p(\Theta^{\infty})$. Thus, there are alternatives:
\begin{equation} \label{C-14}
 \mbox{Either}~R_p^{\infty} = 0, \quad \mbox{or} \quad  \sin(\theta^\infty - \theta_i^{\infty}) = 0, \quad 1 \leq i \leq N. 
\end{equation} 
Next, we present a definition of $(N-k, k)$-type state which clearly satisfies \eqref{C-13}.
\begin{definition} \label{D3.1}
We will say that $\Theta^\infty$ is a $(N-k,k)$-type state if and only if there exists a $\varphi^\infty$ and  $I \subseteq \{1,\cdots,N\} $with $|I| = k > N/2$ such that 
\[ \theta_i^\infty= 
\begin{cases}
\varphi^\infty, \quad \mbox{mod $2\pi \quad $ for $i \in I$} \\
\varphi^\infty + \pi \quad  \mbox{mod $2\pi \quad $ for $i \in I^c$}
\end{cases}
\]
\end{definition}
\begin{remark} Note that for $k = N$, $(N-k, k)$ state corresponds to one-point phase cluster where all phases are the same, 
whereas $(N-k, k)$-state with $k \in \{ [N/2] + 1, \cdots, N-1 \}$ is a bipolar state.
\end{remark}
Now, we recall a characterization of phase-locked states which combines \eqref{C-14} and Definition \ref{D3.1}.  
\begin{proposition} \label{P3.1}
\emph{\cite{BCM}} 
Let $\Theta^{\infty}$ be a phase-locked state for \eqref{C-10}. Then, at least one of following assertions hold. 
\[ \mbox{Either}~R_p^{\infty} = 0 \quad \mbox{or} \quad \Theta^{\infty}~\mbox{is a $(N-k, k)$-type state}. \]
\end{proposition}

Our first main result can be stated in the following theorem. 
\begin{theorem}\label{T3.4} Suppose that the coupling strength $\kappa$ and initial configurations $(\theta_i^{in}, \omega_i^{in})$ satisfy 
\begin{equation*}
\kappa > 0 \quad \mbox{and} \quad \frac{m}{N}\sum_{i=1}^N |\omega^{in}_i|^2 \leq \kappa R_p(\Theta^{in})^2,
\end{equation*}
and let $(\Theta, \Omega)$ be a solution to \eqref{C-10} with initial data $ (\theta_i^{in}, \omega_i^{in})$. Then, the following assertions hold.
\begin{enumerate}
\item
Complete synchronization emerges asymptotically:
\[ \lim_{t \to \infty} |\omega_i(t) - \omega_j(t)| = 0, \quad 1 \leq i, j \leq N.    \]
\item
If $R_p(\Theta^{in}) = 0$, then the system stays as a phase-locked state and 
\[ R_p(\Theta(t)) = 0, \qquad t \geq 0. \]
\item
If $R_p(\Theta^{in}) > 0$, then there exist a positive lower bound $R_*$ for $R_p$ and $(N-k,k)$-type phase $\Theta^{\infty}$ such that 
\[ \inf_{0 \leq t < \infty} R_p(\Theta(t)) \geq R_* \quad \mbox{and} \quad \lim_{t \to \infty} \|\Theta(t) - \Theta^{\infty}\| = 0. \]
\end{enumerate}
\end{theorem}

\vspace{0.2cm}

Next, we turn to the second setting. \newline

\noindent $\bullet$ {\bf Setting B}:  
Assume that 
\[ 
D(M) :=  \max_{1 \leq i,j \leq N} |m_i - m_j| > 0 \quad \mbox{and} \quad D(\gamma) := \max_{1 \leq i,j \leq N} |\gamma_i - \gamma_j| > 0.
\]
Note that  $\theta_i$ satisfies
\[
m_i\dot{\omega}_i = -\gamma_i \omega_i - \kappa R_{p,i} \sin(\theta_i - \varphi_{p,i})\quad \mbox{for } i =1,\cdots,N \mbox{ and } t > 0.
\]
Suppose that $\Theta^{\infty}$ is the equilibrium state starting from the initial data $\Theta^{in}$. Then, it satisfies
\[  R^{\infty}_{p,i} \sin(\theta^{\infty}_i - \varphi^{\infty}_{p,i})  = 0, \quad \mbox{i.e.,} \quad \mbox{either}~~R^{\infty}_{p,i} = 0, \quad \mbox{or} \quad  \sin(\theta^{\infty}_i - \varphi^{\infty}_{p,i}) = 0. \]

\begin{theorem}\label{T3.5}
 Suppose that the network topology, the coupling strength and the initial configuration satisfy the following assumptions: there exist $\bar a > 0$ and $\delta \geq 0$ such that
\begin{align}
\begin{aligned} \label{C-17}
& m_i > 0, \quad \gamma_i > 0, \quad  \sum_{j=1}^N|a_{ij} - \bar a| \leq \delta \quad \mbox{for all} \quad i =1,\cdots,N, \\
& \mbox{and} \quad \sum_{i,j=1}^Na_{ij}\cos(\theta^{in}_i - \theta^{in}_j)  \geq \sum_{i=1}^N m_i |\omega^{in}_i|^2 + 3\delta N + \frac{\delta^2}{\bar a}.
\end{aligned}
\end{align}
Additionally, let $(\Theta, \Omega)$ be a solution to \eqref{C-10} with initial data $\{ (\theta_i^{in}, \omega_i^{in})$. Then, the following assertions hold.
\begin{enumerate}
\item
Complete synchronization emerges asymptotically:
\[ \lim_{t \to \infty} |\omega_i(t) - \omega_j(t)| = 0, \quad 1 \leq i, j \leq N.    \]
\item
If $\displaystyle \sum_{i,j=1}^N a_{ij}\cos(\theta^{in}_i - \theta^{in}_j) = 0$,  then the system stays as a phase-locked state and 
\[ R_{p,i}(\Theta(t)) = 0, \qquad t \geq 0. \]
\item
If $\displaystyle \sum_{i,j=1}^N a_{ij}\cos(\theta^{in}_i - \theta^{in}_j) > 0$, then there exists a positive lower bound $R_*$ for $R_{p,i}$ and $\theta_i(t) - \varphi_{p,i} \to k_i \pi$ as $t \to \infty$ for some $k_i \in \Z$.
\end{enumerate}
\end{theorem}
\begin{remark}
1.~We use the local oder parameter assumption to rewrite \eqref{C-17} as follows.
\[
\sum_{i=1}^N R_{p,i}(0)\cos(\theta^{in}_i - \varphi^{in}_{p,i}) \geq \sum_{i=1}^N m_i\omega_i^2(0) + 3\delta N + \frac{\delta^2}{\bar a}. 
\]
\noindent 2.~The assumptions on the capacity matrix $A = (a_{ij})$ presented in \eqref{C-17} do not require that all entries of that matrix should be strictly positive, i.e., 
the matrix $A$ may have zero entries. \newline

\noindent 3.~If 
\[  a_{ij} = \frac{1}{N}, \quad  m_i=m, \quad \mbox{and} \quad \gamma_i=1 \quad  1 \leq i,j \leq N, \]
then we can choose $\bar a = \kappa/N$ and $\delta = 0$, and this leads to the same conditions in Theorem \ref{T3.4} for the frequency synchronization.
\end{remark}

\subsubsection{A kinetic Kuramoto equation} \label{sec:3.2}
In this part, we summarize our third result on the complete synchronization of the kinetic Kuramoto equation \eqref{B-7}.  For a given initial Radon measure $f^{in}\,d\theta d\omega$, let $P(t)$ be the orthogonal $\theta$-projection of the set $\mbox{supp}(f(\cdot, \cdot, t))$:
\[
D_\theta(f(\cdot, \cdot, t)) := \mbox{diam}(P(t)).
\]
Using this notation, we also set
\[
C_{\ell}(m,f^{in}) := \max\lt\{D_{\theta}(f^{in}),D_{\theta}(f^{in}) + \ell m \frac{d}{dt} \Big|_{t = 0 +} D_\theta(f(t))  \rt\}, \quad \ell = 1, 2. 
\]
We also denote by $\mathcal{P}_2(\T \times \R)$ the set of all probability measures on $\T \times \R$ with finite second moment.

We first present a uniform-in-time stability estimate of measure valued solutions to the kinetic Kuramoto equation \eqref{B-7}.

\begin{theorem} \label{T3.6} Suppose that initial measure, $m$ and $\kappa$ satisfy the assumptions.
\begin{eqnarray*}
&& (i)~\mbox{supp} f^{in}~\mbox{is a compact subset $\T \times \R$} \quad \mbox{and} \quad  0 < C_1(m,f^{in}) < \frac{\pi}{2}.    \cr
&& (ii)~0<  m\kappa \leq \frac{\sin\lt(4C_1(m,f^{in}) \rt)}{2C_1(m,f^{in})}.
\end{eqnarray*}
Then, the following assertions hold:
\begin{enumerate}
\item
There exists a unique measure-valued solution $d\mu_t = f(t) \,d\theta d\omega  \in L^\infty([0,\infty);\mathcal{P}_2(\T  \times \R))$ to the equation \eqref{G-1} with the initial data 
$d\mu_0 = f^{in} d\theta d\omega$. Moreover, 
$d\mu_t = f d\theta d\omega$ can be approximated by the empirical measure $\nu_t = \delta(f^N)$ in the Wasserstein distance uniformly in time:
\[
\limsup_{N \to +\infty} \sup_{t \in [0,\infty)} W_2(f^N(t), f(t)) = 0.
\]
\item
If ${\tilde f}$ is another measure-valued solution to \eqref{G-1} with a compactly supported initial measure ${\tilde f}^{in}$, then there exists a nonnegative constant $C$ independent of $t$ such that
\[
W_2(f(t),{\tilde f}(t)) \leq CW_2(f^{in}, {\tilde f}^{in}) \quad \mbox{for} \quad t \geq 0.
\]
\end{enumerate}
\end{theorem}
We finally state our last theorem on the asymptotic behavior of measure valued solutions to \eqref{B-7}, whose existence is obtained in the previous theorem.

\begin{theorem}\label{T3.7} Let $f \in L^\infty([0,\infty);\pp_2(\T \times \R))$ be a measure valued solution to the equation \eqref{B-7} with initial data:
\bq\label{su_lt}
\frac12\int_{\T \times \R} w^2\, f^{in}(\theta,\omega)\,d\theta d\omega \leq \frac{\kappa R_k(0)^2}{2m}.
\eq
Then we have
\begin{enumerate}
\item
If $R(0) = 0$, then solution stays as a phase-locked state and
\[ R(t) = 0, \quad t > 0. \]
\item
If $R(0) > 0$, then the measure valued solution $d\mu_t = f(t)d\theta d\omega$ tends to a bi-polar state of type $(c_1,c_2)$ asymptotically.
\end{enumerate}

\end{theorem}

\begin{remark}
As can be seen in Section \ref{sec:6}, the above results are lifted from the corresponding particle results presented in Theorem \ref{T3.4} and Theorem \ref{T3.5}. This lifting of particle results via the mean-field limit has been used in earlier results \cite{CCHKK, CHN, HKMP} for the Kuramoto model with and without inertia. 
\end{remark}

%
%

%
%

\section{A Kuramoto ensemble with homogeneous inertia and friction}\label{sec:4}
\setcounter{equation}{0}
In this section, we study the asymptotic phase-locking of homogenous Kuramoto ensemble under setting A:
\[ m_i = m, \quad \gamma_i = 1, \quad \nu_i = 0, \quad \mbox{and} \quad a_{ij} =\frac{1}{N}, \quad 1 \leq i,j \leq N \]
In this situation, system \eqref{main} becomes 
\begin{equation} \label{D-1}
m \ddot{\theta}_i =-\dot{\theta}_i + \frac{\kappa}{N}\sum_{j=1}^N \sin(\theta_j - \theta_i),\quad  i =1,\cdots,N, \quad t > 0,
\end{equation}
or equivalently,
\begin{equation} \label{D-2}
m\ddot{\theta}_i = -\dot{\theta}_i -{\kappa} R_p \sin(\theta_i - \varphi_p),\quad  i =1,\cdots,N, \quad t > 0.
\end{equation}
In the following three subsections, we will provide the proof for three assertions in Theorem \ref{T3.4}.

\subsection{Emergence of CPS}  \label{sec:4.1}
In this subsection, we provide a proof of the first assertion of Theorem \ref{T3.4}. For this, we introduce kinetic and potential energies associated with \eqref{D-1}. 

For a solution $\{ (\theta_i, \omega_i := {\dot \theta}_i) \}$ to \eqref{D-1}, the kinetic energy ${\mathcal E}_K$ and the potential energy ${\mathcal E}_P$ are defined as follows.
\begin{align}
\begin{aligned} \label{D-3}
{\mathcal E}(\Theta, {\dot \Theta}) &:=  {\mathcal E}_K({\dot \Theta})  + {\mathcal E}_{P}(\Theta), \quad {\mathcal E}_K({\dot \Theta}) := \frac{m}{2} \sum_{i=1}^N |\omega_i|^2 \quad \mbox{and}\\
{\mathcal E}_{P}(\Theta) &:= \frac{\kappa}{2N}\sum_{i,j=1}^N(1 - \cos(\theta_j - \theta_i)) = \frac{\kappa N}{2} (1-R_p^2).   
\end{aligned}
\end{align}
Note that both the kinetic and the potential energies are nonnegative. Moreover, the potential energy is uniformly bounded:
\[ \label{D-4}
\sup_{0\leq t < \infty} {\mathcal E}_{P}(\Theta(t)) \leq \frac{\kappa N}{2}. 
\]
\begin{proposition}\label{P4.1} 
\emph{(Energy estimate)}
Let $\Theta := (\theta_1,\cdots,\theta_N)$ be a phase vector whose components are governed by \eqref{D-1}.  Then, the total energy satisfies a dissipation estimate:
\begin{equation} \label{D-5}
 \frac{d}{dt} {\mathcal E}(\Theta, {\dot \Theta}) =-  \frac{2\gamma}{m} {\mathcal E}_K({\dot \Theta}), \quad t > 0. 
\end{equation}
\end{proposition}
\begin{proof} It follows from \eqref{D-1} that 
\begin{equation} \label{D-6}
 m\dot{\omega}_i  = -\gamma \omega_i + \frac{\kappa}{N}\sum_{j=1}^N \sin(\theta_j - \theta_i). 
\end{equation} 
We multiply $\omega_i$ to \eqref{D-6} and sum it over all $i$ to obtain
\begin{align}
\begin{aligned} \label{D-7}
\frac{d}{dt}\sum_{i=1}^N \frac{m \omega_i^2}{2} &= -\gamma \sum_{i=1}^N \omega_i^2 + \frac{\kappa}{N}\sum_{i,j=1}^N\sin(\theta_j - \theta_i)\omega_i\cr
&= -\gamma \sum_{i=1}^N \omega_i^2 + \frac{\kappa}{2N}\sum_{i,j=1}^N\sin(\theta_j - \theta_i)(\omega_i - \omega_j)\cr
&= -\gamma \sum_{i=1}^N \omega_i^2 + \frac{\kappa}{2N}\frac{d}{dt}\sum_{i,j=1}^N\cos(\theta_j - \theta_i).
\end{aligned}
\end{align}
This yields the desired estimate.
\end{proof}
In the sequel, for notational simplicity, we suppress $\Theta$ and ${\dot \Theta}$ dependence in ${\mathcal E},~{\mathcal E}_K$ and ${\mathcal E}_P$, i.e.,
\[  
{\mathcal E}(t) := {\mathcal E}(\Theta(t), {\dot \Theta}(t)), \qquad {\mathcal E}_K(t) :=  {\mathcal E}_K({\dot \Theta}(t)),  \qquad {\mathcal E}_{P} (t):= {\mathcal E}_{P}(\Theta(t)). 
\]
As a direct corollary of Proposition \ref{P4.1}, we will show that the kinetic energy ${\mathcal E}_K$  vanishes asymptotically. 
\begin{corollary}\label{C4.1}
Let $\Theta := (\theta_1,\cdots,\theta_N)$ be a phase vector whose components are governed by \eqref{D-1} and satisfying the relationships:
\begin{equation} \label{D-7-1}
 \theta_c(0) = 0 \quad \mbox{and} \quad \omega_c(0) = 0. 
\end{equation}
Then, the kinetic energy satisfies 
\[ \sup_{0 \leq t < \infty} {\mathcal E}_K(t) \leq \max \lt\{ {\mathcal E}_K(0),  \frac{m^2\kappa^2 N}{4\gamma^2} \rt\} =: {\mathcal E}_K^{\infty}, \quad  \lim_{t \to \infty} {\mathcal E}_K(t) = 0. 
\]
\end{corollary}
\begin{proof} 
(i) Note that Lemma \ref{L2.1} and \eqref{D-7-1} yield
\begin{equation} \label{D-7-2}
\theta_c(t) = 0, \quad \omega_c(t) = 0, \quad t \geq 0. 
\end{equation}
On the other hand, it follows from \eqref{D-7} that 
\begin{equation} \label{D-8}
 \frac{d{\mathcal E}_K}{dt} + \frac{2\gamma}{m} {\mathcal E}_K = - \frac{\kappa}{2N}\sum_{i,j=1}^N\sin(\theta_i - \theta_j)(\omega_i - \omega_j).
\end{equation} 
First, note that the relation \eqref{D-7-2} implies
\begin{equation} \label{D-9}
 \sum_{i,j} |\omega_i - \omega_j|^2 = 2N \sum_{i} |\omega_i|^2 = 4N {\mathcal E}_K.
\end{equation}
Then, \eqref{D-9}, $|\sin^2(\theta_i - \theta_j)|^2 \leq1$, and Cauchy-Schwarz's inequality yield
\begin{align}
\begin{aligned} \label{D-10}
& \lt| \frac{\kappa}{2N}\sum_{i,j=1}^N\sin(\theta_i - \theta_j)(\omega_i - \omega_j) \rt| \\
& \hspace{1cm} \leq \frac{\kappa}{2N} \lt(  \sum_{i,j=1}^N \sin^2(\theta_i - \theta_j)  \rt)^{\frac{1}{2}} \cdot \lt(  \sum_{i,j=1}^N |\omega_i - \omega_j|^2 \rt)^{\frac{1}{2}} \\
& \hspace{1cm} \leq \kappa \sqrt{N} \sqrt{{\mathcal E}_K}.
\end{aligned}
\end{align}
We now combine \eqref{D-8} and \eqref{D-10} to obtain a differential inequality for ${\mathcal E}_K$:
\begin{equation} \label{D-11}
\lt| \frac{d{\mathcal E}_K}{dt} + \frac{2\gamma}{m} {\mathcal E}_K \rt| \leq \kappa \sqrt{N} \sqrt{{\mathcal E}_K}, \quad t > 0.
\end{equation}
Next, we set $Z := \sqrt{{\mathcal E}_K}$, and derive a Gronwall's inequality for $Z$:
\[ 
\frac{dZ}{dt} + \frac{\gamma}{m} Z \leq \frac{\kappa \sqrt{N}}{2}, \quad t > 0.
\]
This implies
\begin{equation} \label{D-12}
Z(t) \leq \lt( \sqrt{{\mathcal E}_0} -  \frac{m\kappa \sqrt{N}}{2\gamma} \rt)  e^{-\frac{\gamma t}{m}} +  \frac{m\kappa \sqrt{N}}{2\gamma}.
\end{equation}
We now consider two cases depending on the relative size between $\sqrt{{\mathcal E}_0}$ and $  (m\kappa \sqrt{N})/(2\gamma)$. \newline

\noindent $\bullet$ Case A: If $\sqrt{{\mathcal E}_0} > (m\kappa \sqrt{N})/(2\gamma)$, then we use $e^{-\frac{\gamma t}{m}} \leq 1$ in \eqref{D-12} to derive
\begin{equation} \label{D-13}
Z(t) \leq \sqrt{{\mathcal E}_K(0)}, \quad \mbox{i.e.,} \quad  {\mathcal E}_K(t) \leq {\mathcal E}_K(0). 
\end{equation}

\vspace{0.2cm}

\noindent $\bullet$ Case B: If $\sqrt{{\mathcal E}_0} \leq(m\kappa \sqrt{N})/(2\gamma)$, then we have
\begin{equation} \label{D-14}
Z(t) \leq \frac{m\kappa \sqrt{N}}{2\gamma}, \quad \mbox{i.e.,} \quad {\mathcal E}_K(t) \leq \frac{m^2\kappa^2 N}{4\gamma^2}.
\end{equation}
Finally, we combine \eqref{D-13} and \eqref{D-14} to derive a uniform bound of ${\mathcal E}_K$. \newline

\noindent (ii) We integrate \eqref{D-5} to obtain
\[
\int_0^t {\mathcal E}_K(s) ds \leq \frac{m}{2\gamma} {\mathcal E}(0) \quad \mbox{for} \quad t \geq 0. 
\]
In order to show that the above integrand tends to zero as $t \to \infty$, we have to show that the integrand is uniformly continuous (see Barbalat's lemma \cite{Ba}). In fact, we can see that 
the integrand has a finite derivative, i.e., it is Lipschitz continuous which implies the uniform continuity.  Note that \eqref{D-11} and the first result (i) yield
\[ 
 \lt| \frac{d{\mathcal E}_K}{dt} \rt| \leq \frac{2\gamma}{m} {\mathcal E}^{\infty}_K + \sqrt{N} \sqrt{{\mathcal E}^{\infty}_K}.
\]
Thus, we obtain the desired result. 
\end{proof}
\begin{remark}
Since $|\omega_i| \leq \sqrt{ \frac{2 {\mathcal E}_k(\Theta)}{m}}$,  the decay of kinetic energy yields 
\[ 
\lim_{t \to \infty} |\omega_i| = 0. 
\]
Thus
\[
\lim_{t \to \infty} |\omega_i(t) - \omega_j(t)| = 0, 
\]
which verifies the first assertion of Theorem \ref{T3.4}. 
\end{remark}
Next, we show that the derivative of frequency, i.e., angular acceleration ${\ddot \theta}_i$ tends to zero as $t \to \infty$ without decay rate. For this, we first recall a generalized Gronwall type lemma as follows.

\begin{lemma}\label{L4.1} 
\emph{\cite{C-H-H-J-K}}
Let $y = y(t)$ be a nonnegative $\mc^1$-function satisfying the following differential inequality:
\begin{equation} \label{B-6}
\dot{y}(t) + \alpha y(t) = \beta(t), \quad \forall~ t > 0, \qquad y(0) = y^{in},
\end{equation}
where $\alpha$ is a positive constant and $\beta$ is a bounded continuous function decaying to zero as $t \to \infty$. Then, $y$ satisfies
\[
y(t) \leq y^{in} e^{-\alpha t} + \frac1\alpha \max_{s \in [t/2,t]}|\beta(s)| + \frac{\|\beta\|_{L^\infty}}{\alpha} e^{-\frac{\alpha t}{2}} \qquad \forall t \geq 0.
\]
In particular, $y$ tends to zero as $t \to \infty$. 
\end{lemma}
\begin{proof} Although the proof can be found in Appendix A of \cite{C-H-H-J-K}, we briefly present its proof here for the reader's convenience.  
We multiply integrating factor $e^{\alpha t}$ to \eqref{B-6} to obtain
\[ y(t) \leq y^{in} e^{-\alpha t} + \int_0^t e^{-\alpha(t-s)} \beta(s) ds. \]
Then, using the decay property of $h$ we get
\begin{align*}
\begin{aligned}
y(t) &\leq y^{in} e^{-\alpha t} + e^{-\alpha t}\lt(\int_0^{t/2} + \int_{t/2}^t \rt)\beta(s)e^{\alpha s}\,ds\cr
&\leq y^{in} e^{-\alpha t} + \|\beta\|_{L^\infty}e^{-\alpha t}\int_0^{t/2} e^{\alpha s}\,ds + e^{-\alpha t}\max_{s \in [t/2,t]}|\beta(s)|\int_{t/2}^t e^{\alpha s}\,ds\cr
&\leq y^{in} e^{-\alpha t} + \frac{1}{\alpha} \max_{s \in [t/2,t]}|\beta(s)| + \frac{\|\beta\|_{L^\infty}}{\alpha} e^{-\frac{\alpha t}{2}}  \to 0 \quad \mbox{as} \quad t \to \infty.
\end{aligned}
\end{align*}
\end{proof}
Now, we are ready to show that ${\dot \omega}_i$ tends to zero as $t \to \infty$, which will be used crucially in the proof of second part in Theorem \ref{T3.1} in next subsection. Recall that ${\dot \omega}_i$ satisfies 
\begin{equation} \label{D-15-1}
\dot{\omega}_i = - \frac{1}{m} \Big[ \gamma \omega_i + {\kappa} R_p \sin(\theta_i - \varphi_p) \Big]\quad \mbox{for } i =1,\cdots,N \mbox{ and } t > 0.
\end{equation}
\begin{proposition}\label{P4.2}
Let $\Theta := (\theta_1,\cdots,\theta_N)$ be a phase vector whose components are governed by \eqref{D-1} and satisfy the zero sum conditions $\theta_c(t) = 0$ and $\omega_c(t) = 0.$
Then, we have
\[ \lim_{t \to \infty} |{\dot \omega}_i(t)|  = 0, \quad 1 \leq i \leq N. \]
\end{proposition}
\begin{proof} From the relation \eqref{D-15-1}, it suffices to show that the right hand side tends to zero. For this, we set 
\begin{equation} \label{D-16}
{\mathcal F}(t) := \frac{1}{2} \sum_{i=1}^N| \gamma \omega_i + \kappa R_p \sin(\theta_i - \varphi_p)|^2. 
\end{equation}
Next, we will derive a Gronwall type differential inequality for ${\mathcal F}$. For this, we use \eqref{B-5} and \eqref{D-15-1} to get
\begin{align}
\begin{aligned} \label{D-17}
&\frac{d}{dt}\lt(\gamma \omega_i + \kappa R_p \sin(\theta_i - \varphi_p) \rt)  \\
& \hspace{1cm} = \gamma {\dot \omega}_i + \kappa \frac{d}{dt} \Big[ R_p \sin(\theta_i - \varphi_p) \Big] \\
& \hspace{1cm} = - \frac{\gamma}{m} \Big[ \gamma \omega_i + {\kappa} R_p \sin(\theta_i - \varphi_p) \Big] + \frac{\kappa}{N} \sum_{j =1}^{N} \cos(\theta_i - \theta_j) (\omega_i - \omega_j).
\end{aligned}
\end{align}
It follows from \eqref{D-16} and \eqref{D-17} that we have
\begin{align}
\begin{aligned} \label{D-18}
\frac{d{\mathcal F}}{dt} &= \frac{1}{2}\frac{d}{dt}\sum_{i=1}^N| \gamma \omega_i + \kappa R_p \sin(\theta_i - \varphi_p)|^2 \\
 &=\sum_{i=1}^N \lt( \gamma \omega_i + \kappa R_p \sin(\theta_i - \varphi_p) \rt)\frac{d}{dt}\lt( \gamma \omega_i + \kappa R_p \sin(\theta_i - \varphi_p) \rt) \\
 &=   - \frac{\gamma}{m} \sum_{i=1}^N | \gamma \omega_i + {\kappa} R_p \sin(\theta_i - \varphi_p)|^2 \\
 & + \frac{\kappa}{N} \sum_{i, j =1}^{N} \lt( \gamma \omega_i + \kappa R_p \sin(\theta_i - \varphi_p) \rt) \cos(\theta_i - \theta_j) (\omega_i - \omega_j) \\
 &=:   - \frac{2 \gamma}{m} {\mathcal F} + {\mathcal I}_{1}.
\end{aligned}
\end{align}
We further refine the term ${\mathcal I}_1$ as follows. 
\begin{align}
\begin{aligned} \label{D-19}
 {\mathcal I}_{1} &=\frac{\kappa \gamma}{N} \sum_{i,j} \cos(\theta_i - \theta_j) \omega_i (\omega_i - \omega_j) + \frac{\kappa^2 R_p}{N} \sum_{i,j} \sin(\theta_i - \varphi) \cos(\theta_i - \theta_j) (\omega_i - \omega_j) \\
&= -\frac{\kappa \gamma}{2N} \sum_{i,j} \cos(\theta_i -\theta_j) |\omega_i - \omega_j|^2 + \frac{\kappa^2 R_p}{N} \sum_{i,j} \sin(\theta_i - \varphi) \cos(\theta_i - \theta_j) (\omega_i - \omega_j) \\
& =: {\mathcal I}_{11} + {\mathcal I}_{12}.
\end{aligned}
\end{align}
Below, we estimate the terms ${\mathcal I}_{1i},~i=1,~2$ separately. 

\vspace{0.5cm}

\noindent $\diamond$ (Estimate of ${\mathcal I}_{11}$): We use \eqref{D-9} to obtain
\begin{equation} \label{D-20}
 | {\mathcal I}_{11} | \leq 2 \kappa \gamma {\mathcal E}_K. 
\end{equation} 

\noindent $\diamond$ (Estimate of ${\mathcal I}_{12}$): We use \eqref{D-9} to obtain
\begin{equation} \label{D-21}
 | {\mathcal I}_{12} | \leq \frac{\kappa^2 R_p}{N} \sum_{i,j} |\omega_i - \omega_j| \leq  \kappa^2 R_p \Big( \sum_{i,j} |\omega_i - \omega_j|^2 \Big)^{\frac{1}{2}}  \leq 2 \kappa^2 R_p \sqrt{N }\sqrt{{\mathcal E}_K}.
\end{equation} 
In \eqref{D-18}, we combine \eqref{D-19}, \eqref{D-20}, and \eqref{D-21} to obtain
\begin{equation} \label{D-22}
\frac{d{\mathcal F}}{dt} \leq -\frac{2\gamma}{m} {\mathcal F} + 2 \kappa \gamma {\mathcal E}_K  +  2 \kappa^2 R_p \sqrt{N }\sqrt{{\mathcal E}_K}, \quad t > 0.
\end{equation}
We apply Lemma \ref{L4.1}  for \eqref{D-22} with 
\[ \alpha = \frac{2\gamma}{m} \quad \mbox{and} \quad \beta(t) :=  2 \kappa^2 R_p \sqrt{N }\sqrt{{\mathcal E}_K} \]
to obtain the desired zero convergence of ${\mathcal F}$ as $t \to \infty$. 
\end{proof}
Suppose that coupling strength and initial data $(\Theta^{in}, {\dot \Theta}^{in})$ satisfy
\begin{equation} \label{D-26}
\kappa R_p^2(0) \geq \frac{m}{N} \sum_{i=1}^N |\omega^{in}_i|^2 = \frac{2}{N} {\mathcal E}_K(0), \qquad R_p(0) := R_p(\Theta^{in}).
\end{equation}
Depending on $R_p^0$, we consider the following two cases:
\[ \mbox{Either}~~R_p(0) = 0 \quad \mbox{or} \quad R_p(0) > 0. \]

\subsection{Asymptotic phase-locking I} \label{sec:4.2} We next study the relaxation dynamics of the phase vector $\Theta$.  Suppose that 
\[ R_p(0) = 0. \]
In this case, thank to the relations \eqref{B-5} and \eqref{D-26}, we have
\[ \omega_i^{in} = 0, \quad  \sum_{j=1}^N \sin(\theta_j^{in}  - \theta_i^{in}) = 0, \quad  1 \leq i \leq N. \]
It follows from \eqref{D-1} that
\[ {\dot  \theta}_i(0+) = 0,  \quad  {\dot \omega}_i(0+) = -\frac{1}{m} \Big[ \gamma \omega_i^{in} - \frac{\kappa}{N}\sum_{j=1}^N \sin(\theta_j^{in} - \theta_i^{in}) \Big] = 0. \]
Then, by the uniqueness of ODE theory, 
\[ \theta_i(t) = \theta_i^{in} \quad \mbox{for} \quad t \geq 0. \]
Thus, we have
\[ R_p(t) = R_p(0) = 0. \quad \mbox{for all~~$t \geq 0$.} \] 
By Proposition \ref{P3.1}, $\Theta(t)$ is a phase-locked state. 

\subsection{Asymptotic phase-locking II} \label{sec:4.3}  Suppose that 
\begin{equation} \label{D-26-1}
 R_p(0) > 0.
\end{equation} 
In this case, we claim:  there exists an emergent phase-locked state $\Theta^{\infty}$ issued from the initial data $R_0$ along the Kuramoto flow, i.e., $\Theta^{\infty}$ is a $(N-k, k)$-type phase-locked state. It follows from Proposition \ref{P4.2} that there exists a bounding function $g(t)$ such that
\[ \sum_{i=1}^N| \gamma \omega_i + \kappa R_p \sin(\theta_i - \varphi_p)|^2 \leq g(t) \quad \mbox{and} \quad  \lim_{t \to \infty} g(t)= 0. \]

\vspace{0.2cm}

\noindent $\diamond$ Step A: First, we show
\begin{equation} \label{D-27}
 \lim_{t \to \infty} R_p^2(t) \sin^2(\theta_i(t) - \varphi_p(t)) = 0, \quad 1 \leq i \leq N. 
\end{equation} 
It follows from Corollary \ref{C4.1} and Proposition \ref{P4.2} that we have
\begin{align*}
\begin{aligned} 
& \kappa^2 \sum_{i=1}^N R_p^2\sin^2(\theta_i - \varphi_p) = \sum_{i=1}^N|\gamma \omega_i + \kappa R_p \sin(\theta_i - \varphi_p) - \gamma \omega_i |^2 \\
& \hspace{0.5cm} \leq 2\sum_{i=1}^N|\omega_i + \kappa R_p \sin(\theta_i - \varphi_p)|^2 + 2\gamma \sum_{i=1}^N \omega_i^2 \leq 2g(t) + \frac{4\gamma}{m} {\mathcal E}_K(t) \to 0 \quad \mbox{as} \quad t \to \infty.
\end{aligned}
\end{align*}
This establishes \eqref{D-27}. 

\vspace{0.2cm}

\noindent $\diamond$ Step B: we claim that there exists a $R_* > 0$ such that 
\[  R_p(t) \geq R_* \quad \mbox{for all $t \geq 0$}. \]
Suppose not, i.e., there exists $t_0 \in (0, \infty]$ such that 
\begin{equation} \label{D-29}
\lim_{t \to t_0 -} R_p(t) = 0. 
\end{equation}
We integrate the energy relation \eqref{D-5} from $0$ to $t$ to obtain
\begin{equation} \label{D-29-1}
{\mathcal E}_K(t) + \frac{\kappa N}{2} (1-R_p^2(t)) + \frac{2}{m}  \int_0^t {\mathcal E}_K(s) ds = {\mathcal E}_K(0) + \frac{\kappa N}{2} (1-R_p^2(0)), \quad t \geq 0. 
\end{equation}
Letting $t \to t_0-$ in \eqref{D-29-1} and using the relation \eqref{D-3}, we find that 
\begin{align}
\begin{aligned} \label{D-30}
{\mathcal E}_K(0) + \frac{\kappa N}{2} (1-R_p^2(0)) &=    {\mathcal E}_K(t_0) + \frac{\kappa N}{2} (1-R_p^2(t_0)) + \frac{2}{m} \int_0^{t_0} {\mathcal E}_K(s) ds \\
&=  {\mathcal E}_K(t_0) +  \frac{\kappa N}{2} + \frac{2}{m} \int_0^{t_0} {\mathcal E}_K(s) ds. 
\end{aligned}
\end{align}
On the other hand, it follows from the assumption \eqref{D-26} that the R.H.S. of \eqref{D-30} can be estimated as follows.
\begin{equation} \label{D-31}
\frac{\kappa N}{2} \geq \frac{\kappa N}{2}  + {\mathcal E}_K(0) - \frac{\kappa N}{2} R_p^2(0).
\end{equation}
We combine \eqref{D-30} and \eqref{D-31} to obtain
\begin{equation} \label{D-32}
0 \geq  {\mathcal E}_K(t_0) + \frac{2}{m} \int_0^{t_0} {\mathcal E}_K(s) ds. 
\end{equation}
Since ${\mathcal E}_K(t)$  is $\mc^1$ and non-negative, the relation \eqref{D-32} yields
\begin{equation} \label{D-33}
  {\mathcal E}_K(t) = 0, \quad \forall~t \in [0,t_0]. 
\end{equation} 
Finally, we combine \eqref{D-29}, \eqref{D-29-1} and \eqref{D-33} to obtain
\[
0 = R_p(t_0)= R_p(0), \quad t \in [0, t_0], 
\]
which is contradictory to the assumption \eqref{D-26-1}. Hence, there exists a constant $R_* > 0$ such that 
\[ R_p(t) \geq R_* > 0, \quad t \geq 0. \]

\noindent $\diamond$ Step C: It follows from \eqref{D-2} that we have
\[ m\dot{\omega}_i = -\gamma \omega_i -{\kappa} R_p \sin(\theta_i - \varphi_p). \]
Then, we let $t \to \infty$ and use $R_p > R_* > 0$, the second result in Corollary \ref{C4.1}, and Proposition \ref{P4.2} to obtain
\[  \lim_{t \to \infty} \sin(\theta_i(t) - \varphi_p(t)) = 0 \quad \mbox{for} \quad 1 \leq i \leq N.  \]
Since the sinusoidal function has isolated zeros, we get 
\bq\label{D-34}
\theta_i(t) - \varphi(t) \to k_i \pi \quad \mbox{for some} \quad k_i \in \mathbb{Z}.
\eq
On the other hand, it follows from \eqref{D-34} that
\[
\varphi(t) = -\frac1N \sum_{i=1}^N(\theta_i(t) - \varphi(t)) \to - \frac1N\sum_{i=1}^N k_i\pi \quad \mbox{as} \quad t \to \infty,
\]
i.e., $\varphi(t)$ converges to some constant $\varphi^*$ as $t$ goes to infinity. This implies that
\[
\theta_i(t) = \theta_i(t) - \varphi(t) + \varphi(t) \to k_i\pi + \varphi^* \quad \mbox{as} \quad t \to \infty.
\]
The relation \eqref{B-4-1} yields
\[
R_p(t) = \frac1N\sum_{j=1}^N\cos(\theta_j(t) - \varphi_p(t)) \to \frac1N\sum_{j=1}^N\cos(k_j\pi) =: R^* > 0.
\]
This completes the proof.
\begin{remark} It follows from \eqref{D-34} that
\[
\max_{1\leq i,j \leq N}|\theta_i(t) - \theta_j(t)| \leq 2\max_{1\leq i \leq N}|\theta_i(t) - \varphi(t)| \leq C\max_{1 \leq i \leq N} k_i, \quad t \gg 1,
\]
where $k_i$ is defined in \eqref{D-34}. Then, by the second-order gradient-like structure of \eqref{D-1}, we can also show that convergence of $\Theta$ to the phase-locked solution $\Theta^\infty$ of \eqref{D-1}.
\end{remark}
%
%

%
%

\section{A Kuramoto ensemble with heterogeneous inertia and friction} \label{sec:5}
\setcounter{equation}{0}
In this section, we study the emergence of asymptotic phase-locking for the heterogeneous Kuramoto ensemble on symmetric network under the effect of inhomogeneous inertia and zero natural frequency:
\begin{equation}\label{F-1}
m_i\ddot{\theta}_i + \gamma_i\dot{\theta}_i =  \kappa \sum_{j=1}^Na_{ij} \sin(\theta_j - \theta_i)\quad \mbox{for } i =1,\cdots,N \mbox{ and } t > 0.
\end{equation}
This system appears in the modeling of a lossless network-reduced power system \cite{CL}. Here, we assume $a_{ij} = a_{ji} \geq 0$. We refer the reader to \cite{CHLXY, CL} for more detailed discussion on this system. 

Note that \eqref{F-1} can be rewritten as a system of the first-order ODEs:
\begin{align*}
\begin{aligned}
{\dot \theta}_i &= \omega_i, \quad i=1, 2, \dots, N, \quad t > 0, \\
{\dot \omega}_i &= \frac{1}{m_i} \left(  -\gamma_i\omega_i +  \kappa \sum_{j=1}^{N}
a_{ij}\sin(\theta_j - \theta_i) \right).
\end{aligned}
\end{align*}
As mentioned in Section \ref{sec:2.3}, system \eqref{F-1} can be rewritten as
\[
m_i\ddot{\theta}_i = -\gamma_i\dot{\theta}_i - \kappa R_{p,i} \sin(\theta_i - \varphi_{p,i})\quad \mbox{for } i =1,\cdots,N \mbox{ and } t > 0.
\]
In the following two subsections, we proceed arguments as in Section \ref{sec:4}. 

\subsection{Emergence of CSP}  \label{sec:5.1}
In this subsection, we provide the proof of the first assertion in Theorem \ref{T3.5} on the emergence of complete synchronization. Note that once we prove that
\begin{equation} \label{F-4-1} 
\lim_{t \to \infty} |\omega_i(t)| = 0 \quad \mbox{for} \quad i =1, \cdots, N, 
\end{equation}
then we obtain the complete synchronization estimate
\[ \lim_{t \to \infty} |\omega_i(t) - \omega_j(t)| \leq  \lim_{t \to \infty} |\omega_i(t)| +  \lim_{t \to \infty} |\omega_j(t)| = 0 \quad \mbox{for} \quad 1 \leq i, j \leq N. \]
Thus, our main goal is to verify the estimate \eqref{F-4-1} using the energy method. For a solution $\{ (\theta_i, \omega_i := {\dot \theta}_i) \}$ to \eqref{F-1} we define the kinetic and the potential energies as follows.
\begin{align*}
\begin{aligned} 
{\tilde {\mathcal E}}(\Theta, {\dot \Theta}) &:=  {\tilde {\mathcal E}}_K({\dot \Theta})  + {\tilde {\mathcal E}}_{P}(\Theta), \quad {\tilde {\mathcal E}}_K({\dot \Theta}) := \frac{1}{2} \sum_{i=1}^N m_i |\omega_i|^2, \\
{\tilde {\mathcal E}}_{P}(\Theta) &:= \frac{\kappa}{2}\sum_{i,j=1}^N a_{ij} (1 - \cos(\theta_i - \theta_j)).  
\end{aligned}
\end{align*}
Note that in the homogeneous case and all-to all network topology, i.e., $m_i = m$ and $a_{ij} = \frac{1}{N}$, the above energies reduce to the one defined in \eqref{D-3}.  For notational simplicity, we suppress $\Theta$ and ${\dot \Theta}$ dependence in
 ${\tilde {\mathcal E}}_K$ and ${\tilde {\mathcal E}}_P$, i.e.,
\[  {\tilde {\mathcal E} }:= {\tilde {\mathcal E}}(\Theta, {\dot \Theta}), \qquad {\tilde {\mathcal E}}_K := {\tilde  {\mathcal E}}_K({\dot \Theta}),  \qquad {\tilde  {\mathcal E}}_{P} :={\tilde {\mathcal E}}_{P}(\Theta). \]

\noindent In the following Lemma, we compute the time derivative of the total energy. 
\begin{proposition}\label{P5.1} 
Let $\Theta := (\theta_1,\cdots,\theta_N)$ be a phase vector whose components are a solution to \eqref{F-1} with $\nu_i = 0$ for all $i$ in $\{1,\cdots,N\}$. Then, we have
\[ \label{F-6}
\frac{d}{dt}  {\tilde {\mathcal E}}(\Theta, {\dot \Theta}) =  -\sum_{i=1}^N \gamma_i\omega_i^2, \quad t > 0.
\]
\end{proposition}
\begin{proof} We use 
\begin{equation*} 
 m_i\dot{\omega}_i = -\gamma_i \omega_i + \kappa \sum_{j=1}^Na_{ij} \sin(\theta_j - \theta_i),
 \end{equation*}
and the same arguments in Proposition \ref{P4.1} to obtain
\begin{align*}
\begin{aligned}
\frac{d}{dt} {\tilde {\mathcal E}}_K({\dot \Theta}) &= \frac12\frac{d}{dt}\sum_{i=1}^N m_i\omega_i^2 =\sum_{i=1}^N \omega_i (m_i \dot{\omega}_i) \\
& = -\sum_{i=1}^N \gamma_i \omega_i^2 - \kappa \sum_{i,j=1}^Na_{ij}\sin(\theta_i - \theta_j)\omega_i = -\sum_{i=1}^N \gamma_i \omega_i^2  +  \kappa \sum_{i,j=1}^Na_{ij}\sin(\theta_i - \theta_j)\omega_j \\
&= -\sum_{i=1}^N \gamma_i \omega_i^2 + \frac{\kappa}{2} \frac{d}{dt}\sum_{i,j=1}^N a_{ij}\cos(\theta_i - \theta_j)  =-\sum_{i=1}^N \gamma_i \omega_i^2 -  \frac{d}{dt} {\tilde {\mathcal E}}_{P}(\Theta).
\end{aligned}
\end{align*}
Here, we have used the assumption that  $a_{ij} = a_{ji}$. 
\end{proof}
We also introduce extremal parameters:
\[
d_u := \max_{1 \leq i\leq N}d_i, \quad   d_\ell := \min_{1 \leq i\leq N}d_i, \quad m_u := \max_{1 \leq i\leq N}m_i, \quad \mbox{and} \quad m_\ell := \min_{1 \leq i\leq N}m_i.
\]
As a direct corollary of Proposition \ref{P5.1}, we will show that the kinetic energy vanishes asymptotically, i.e., the frequencies tend to zero as $t \to \infty$.

\begin{corollary}
Let $\Theta := (\theta_1,\cdots,\theta_N)$ be a phase vector whose components are governed by \eqref{F-1}.
Then, we have
\[ \lim_{t \to \infty} |\omega_i(t)|  = 0. \]
\end{corollary}
\begin{proof} As in the proof of Corollary \ref{C4.1}, it suffices to show that the time-rate of change of the energy dissipation $\sum_{i=1}^N \gamma_i \omega_i^2$ is uniformly bounded.  By a direct computation, we get
\begin{align*}
\begin{aligned}
\lt|\frac{d}{dt} \sum_{i=1}^N \gamma_i\omega_i^2 \rt|&\leq 2\lt| \sum_{i=1}^N \frac{\gamma_i^2}{m_i}\omega_i^2\rt| + 2\lt|\sum_{i,j=1}^N \frac{\gamma_i}{m_i}a_{ij}\sin(\theta_j - \theta_i)\omega_i \rt|\cr
&\leq 2 \sum_{i=1}^N \frac{\gamma_i^2}{m_i}\omega_i^2 + 2 N\|A\|_\infty\sqrt{\sum_{i=1}^N\frac{\gamma_i^2}{m_i}}\sqrt{\sum_{i=1}^N m_i\omega_i^2}\cr
&\leq 4\frac{\gamma_u^2}{m_\ell^2} {\tilde {\mathcal E}}_K(t) + 2\sqrt{2} \|A\|_\infty N\sqrt{N}\frac{\gamma_u}{\sqrt{m_\ell}}\sqrt{{\tilde {\mathcal E}}_K(t)}\cr
&\leq 4\frac{\gamma_u^2}{m_\ell^2} {\tilde {\mathcal E}}_K(0) + 2\sqrt{2} \|A\|_\infty N\sqrt{N}\frac{\gamma_u}{\sqrt{m_\ell}}\sqrt{{\tilde {\mathcal E}}_K(0)}\cr
& < \infty \quad \mbox{uniformly in $t$}.
\end{aligned}
\end{align*}
Here, $\|A\|_\infty := \max_{ij}a_{ij}$ and we used the fact that 
\[
\lt|\sum_{i,j=1}^N \frac{\gamma_i}{m_i}a_{ij}\sin(\theta_j - \theta_i)\omega_i \rt| \leq N\|A\|_\infty \lt|\sum_{i=1}^N \frac{\gamma_i}{m_i}\omega_i \rt| \leq N\|A\|_\infty\sqrt{\sum_{i=1}^N\frac{\gamma_i^2}{m_i}}\sqrt{\sum_{i=1}^N m_i\omega_i^2},
\]
and the monotonicity of the total energy. Thus, by the same argument in the proof of Corollary \ref{C4.1}, we have
\[ \lim_{t \to \infty} \sum_{i=1}^N \gamma_i \omega_i^2(t) = 0, \quad \mbox{i.e.,} \quad \lim_{t \to \infty} |\omega_i(t)| = 0.   \]
This yields the desired estimate. 
\end{proof}
Next, we study the asymptotic convergence to zero of angular acceleration. For this, we set
\[ {\mathcal F}(t) := \frac{1}{2} \sum_{i=1}^N| \gamma_i \omega_i + \kappa R_p \sin(\theta_i - \varphi_{p,i})|^2. \]

\begin{proposition}\label{P5.2}
Let $\Theta := (\theta_1,\cdots,\theta_N)$ be a phase vector whose component is a solution to \eqref{F-1} with $\nu_i = 0$. Then, we have
\[ \lim_{t \to \infty} |{\dot \omega}_i(t)|  = 0, \quad 1 \leq i \leq N. \]
\end{proposition}
\begin{proof} We use 
\[  \gamma_i {{\dot \omega}}_i = -\frac{\gamma_i}{m_i} (\gamma_i \omega_i + \kappa R_{p,i} \sin(\theta_i - \varphi_{p,i} ))  \]
to find
\begin{align}
\begin{aligned} \label{F-8}
&\frac12\frac{d}{dt} \sum_{i=1}^N |\gamma_i\omega_i + \kappa R_{p,i}\sin(\theta_i - \varphi_{p,i})|^2 \\
&\hspace{1cm} = \sum_{i=1}^N \lt( \gamma_i\omega_i + \kappa R_{p,i}\sin(\theta_i - \varphi_{p,i} \rt)\lt(\gamma_i \dot\omega_i + \frac{d}{dt}\lt( \kappa R_{p,i}\sin(\theta_i - \varphi_{p,i}) \rt)\rt) \\
& \hspace{1cm} = - \sum_{i=1}^N \frac{\gamma_i}{m_i} | \gamma_i\omega_i + \kappa R_{p,i}\sin(\theta_i - \varphi_{p,i}) |^2  \\
& \hspace{1.5cm} +  \sum_{i=1}^N \lt( \gamma_i\omega_i + \kappa R_{p,i} \sin(\theta_i - \varphi_{p,i} \rt) \lt( \frac{d}{dt}\lt( \kappa R_{p,i} \sin(\theta_i - \varphi_{p,i}) \rt)\rt) \\
&\hspace{1cm}  =:  {\mathcal I}_{21} + {\mathcal I}_{22}.
\end{aligned}
\end{align}

\noindent $\bullet$ (Estimate of ${\mathcal I}_{21} $):  By direct calculation, we have
\begin{equation} \label{F-9}
{\mathcal I}_{21} \leq -\frac{\gamma_\ell}{m_u}\sum_{i=1}^N |\gamma_i \omega_i + \kappa R_{p,i} \sin(\theta_i - \varphi_{p,i})|^2.
\end{equation}

\noindent $\bullet$ (Estimate of ${\mathcal I}_{22} $): We use the identity \eqref{F-3-2}:
\[
\kappa R_{p,i} \sin(\theta_i - \varphi_{p,i}) = \kappa \sum_{j=1}^N a_{ij}\sin(\theta_j - \theta_i)
\]
to find 
\[ \frac{d}{dt} \Big( \kappa R_{p,i} \sin(\theta_i - \varphi_{p,i})  \Big) = \kappa \sum_{j=1}^N a_{ij}\cos(\theta_j - \theta_i) (\omega_j- \omega_i). \]
Note that 
\begin{align*}
\begin{aligned} 
{\mathcal I}_{22} &= \sum_{i=1}^N \lt( \gamma_i\omega_i + \kappa R_{p,i} \sin(\theta_i - \varphi_{p,i} \rt)\lt( \frac{d}{dt}\lt( \kappa R_{p,i} \sin(\theta_i - \varphi_{p,i}) \rt)\rt)  \\
&= \kappa \sum_{i,j=1}^N \gamma_i a_{ij}\cos(\theta_j - \theta_i) \omega_i (\omega_j- \omega_i) \\
&\quad + \kappa^2 \sum_{i,j=1}^{N} R_{p,i} a_{ij}  \sin(\theta_i - \varphi_{p,i}) \cos(\theta_j - \theta_i) (\omega_j- \omega_i).
\end{aligned}
\end{align*}
Thus, we obtain
\begin{align}
\begin{aligned}  \label{F-10}
|{\mathcal I}_{22}| &\leq  \kappa \|A\|_\infty \sum_{i,j=1}^N (\gamma_i |\omega_i| +\kappa)  \cdot |\omega_j- \omega_i|  \\
&\leq \kappa \|A\|_\infty\lt(\sum_{i,j=1}^N \lt(\gamma_i|\omega_i| + \kappa \rt)^2 \rt)^{1/2}\lt(\sum_{i,j=1}^N|\omega_j - \omega_i|^2 \rt)^{1/2} \\
&\leq \kappa \|A\|_\infty \lt( 4N\frac{\gamma_u^2}{m_\ell} {\tilde {\mathcal E}}_K(t) + 2N^2 \kappa^2 \rt)^{1/2}\lt(\frac{4N}{m_\ell}{\tilde {\mathcal E}}_K (t) \rt)^{1/2} \to 0 \quad \mbox{as} \quad t \to \infty.
\end{aligned}
\end{align}
In \eqref{F-8}, we combine \eqref{F-9} and \eqref{F-10} to obtain
\begin{align}
\begin{aligned} \label{F-11}
\frac{d{\mathcal F}(t)}{dt} \leq -\frac{2\gamma_\ell}{m_u} {\mathcal F}(t) +  \kappa \|A\|_\infty \lt( 4N\frac{\gamma_u^2}{m_\ell} {\tilde {\mathcal E}}_K(t) + 2N^2 \kappa^2 \rt)^{1/2}\lt(\frac{4N}{m_\ell}{ \tilde {\mathcal E}}_K (t) \rt)^{1/2}.
\end{aligned}
\end{align}
Then, we apply Lemma \ref{L4.1}  for \eqref{F-11} with
\[ \alpha = \frac{2\gamma_\ell}{m_u}, \quad \beta(t) :=   \kappa \|A\|_\infty \lt( 4N\frac{\gamma_u^2}{m_\ell} {\tilde {\mathcal E}}_K(t) + 2N^2 \kappa^2 \rt)^{1/2}\lt(\frac{4N}{m_\ell}{ \tilde {\mathcal E}}_K (t) \rt)^{1/2}    \]
to obtain the desired estimate. 
\end{proof}
Now, we are in a position to state the main theorem of this section. Its proof is based on a perturbation argument.  Next, we assume that the constants 
$d_u, d_\ell, m_u, m_\ell$ are contained in $(0,\infty)$ and that there exist $\bar a > 0$ and $\delta \geq 0$ such that for $i = 1, \cdots, N$, 
\begin{equation} \label{F-12}
\sum_{j=1}^N|a_{ij} - \bar a| \leq \delta \quad \mbox{and} \quad \sum_{i,j=1}^Na_{ij}\cos(\theta^{in}_i - \theta^{in}_j)  \geq \sum_{i=1}^N m_i |\omega^{in}_i|^2 + 3\delta N + \frac{\delta^2}{\bar a}.
\end{equation}

\subsection{Asymptotic phase-locking I} \label{sec:5.3}  Suppose that we have
\[ \sum_{i,j=1}^N a_{ij}\cos(\theta^{in}_i - \theta^{in}_j) = 0. \]
This and \eqref{F-12} imply
\[ \omega^{in}_i = 0, \quad a_{ij} = {\bar a}, \quad 1 \leq i, j \leq N.      \]
Since 
\[ 
R_{p,i} (t) = {\bar a} N R_p(t) \quad \mbox{ for all $i$ in $\{1,\cdots,N\}$},  
\]
we have
\[ R_p(0) = 0. \]
This yields
\[ R_p(t) = R_p(0) = 0, \quad t \geq 0. \]

\subsection{Asymptotic phase-locking II} \label{sec:5.4}
Suppose that we have
\[ \sum_{i,j=1}^Na_{ij}\cos(\theta^{in}_i - \theta^{in}_j) > 0. \]
By the monotonicity of energy
\begin{align*}
\begin{aligned}
& \frac{\kappa}{2}\sum_{i,j=1}^N a_{ij}  \cos(\theta^{in}_i - \theta^{in}_j)  \\
& \hspace{1cm} < \frac{1}{2} \sum_{i=1}^N m_i |\omega_i(t)|^2  + \frac{\kappa}{2}\sum_{i,j=1}^N a_{ij}  \cos(\theta^{in}_i - \theta^{in}_j) \\
& \hspace{1cm} \leq   \frac{1}{2} \sum_{i=1}^N m_i |\omega^{in}_i|^2  + \frac{\kappa}{2}\sum_{i,j=1}^N a_{ij} \cos(\theta_i(t) - \theta_j(t)).
\end{aligned}
\end{align*}
This again yields
\begin{equation} \label{F-13}
{\mathcal J}_0 :=  \frac{\kappa}{2}\sum_{i,j=1}^N a_{ij}  \cos(\theta^{in}_i - \theta^{in}_j) - \frac{1}{2} \sum_{i=1}^N m_i |\omega^{in}_i|^2 <  \frac{\kappa}{2}\sum_{i,j=1}^N a_{ij} \cos(\theta_i(t) - \theta_j(t)).
\end{equation}
By the assumption on the matrix $A$ and the relation \eqref{B-4-2}, we have
\begin{align}
\begin{aligned} \label{F-14}
\sum_{i,j=1}^Na_{ij}\cos(\theta_i(t) - \theta_j(t)) &=\bar a N^2 R_p^2(t) + \sum_{i,j=1}^N\lt(a_{ij} - \bar a\rt)\cos(\theta_i(t) - \theta_j(t))\cr
&\leq \bar a N^2 R_p^2(t) + \delta N.
\end{aligned}
\end{align}
We combine \eqref{F-13} and \eqref{F-14} to obtain
\bq\label{F-15}
R_p^2(t) > \frac{1}{\bar a N^2} \lt( {\mathcal J}_0 - \delta N\rt) > 0 \quad \mbox{for} \quad t \geq 0.
\eq
Next, we establish a  relation between $R_{p,i}$ and $R_p$. We use the identity $R_{p,i}^2  = \lal R_{p,i} e^{i\varphi_i}, R_{p,i} e^{i\varphi_i} \ral$ and \eqref{F-15} to find
\begin{align}
\begin{aligned}  \label{F-15-1}
R_{p,i}^2 &= \sum_{j,k=1}^N\lt\lal a_{ij}e^{i\theta_j}, a_{ik}e^{i\theta_k}  \rt\ral\cr
&=\sum_{j,k=1}^N\lt(\bar a + a_{ij} - \bar a\rt)\lt(\bar a + a_{ik} - \bar a \rt)\lt\lal e^{i\theta_j}, e^{i\theta_k}\rt\ral \cr
&\geq \bar a^2N^2 R_p^2(t) - 2\bar a\delta N - \delta^2 \cr
&>\bar a({\mathcal J}_0 - 3\delta N) - \delta^2 >0, \quad \mbox{for all} \quad i=1,\cdots,N,
\end{aligned}
\end{align}
This implies 
\begin{equation} \label{F-16}
\inf_{0 \leq t < \infty} R_{p,i}(t)\geq \sqrt{\bar a({\mathcal J}_0 - 3\delta N) - \delta^2 }, \quad 1 \leq i \leq N.
\end{equation}
On the other hand, it follows from Propositions \ref{P4.1} and \ref{P4.2} that we have
\begin{equation} \label{F-17}
 \lim_{t \to \infty} \sum_{i=1}^{N} R_{p,i}^2 \sin^2(\theta_i - \varphi_{p,i}) = 0, \quad i = 1, \cdots, N.    
\end{equation} 
Finally, we combine \eqref{F-16} and \eqref{F-17} to obtain 
\[
\lim_{t \to \infty} \sum_{i=1}^N \sin^2(\theta_i(t) - \varphi_{p,i}(t)) = 0,
\]
i.e., $\theta_i(t) - \varphi_{p,i}(t) \to k_i \pi$ as $t \to \infty$ for some $k_i \in \Z$. 

As a direct application of Theorem \ref{T3.5}, we have the following estimates.
\begin{corollary} Under the same assumptions as in Theorem \ref{T3.5}, we have
\begin{enumerate}
\item
$R_{p,i}$ is majorized by $R_p$: it holds 
\[
R_{p,i}(t) \leq C_0 R_p(t) \quad \mbox{for} \quad t \geq 0,
\]
where $C_0 > 0$ is given by
\[
C_0 := \sqrt{\bar a^2 N^2 + \frac{\bar a N^2(2\bar a \delta N + \delta^2)}{{\mathcal J}_0 - \delta N}}.
\]
\item 
A relation between the local and global average phases of oscillators: it holds
\bq\label{F-19}
\cos(\varphi_{p,i} - \varphi_p) \geq \frac{\bar a N^2}{C_0} - \frac{C_0 \delta N}{\bar a( {\mathcal J}_0 - 3\delta N) - \delta^2},
\eq
where ${\mathcal J}_0 > 0$ appeared in \eqref{F-13}.
\end{enumerate}
\end{corollary}
\begin{proof} (i) Using a similar estimate as in \eqref{F-15-1}, we get
\[
R^2_{p,i}  \leq \bar a^2N^2 R_p^2(t) + 2\bar a\delta N + \delta^2.
\]
This together with \eqref{F-15} provides the desired inequality. 

\noindent (ii) By direct calculation, we have 
\begin{align}\label{est_cos}
\begin{aligned}
\cos(\varphi_{p,i} - \varphi_p) &= \frac{1}{R_p R_{p,i}}\lt\lal R_p e^{i\varphi_p}, R_{p,i} e^{i\varphi_{p,i}}\rt\ral 
=\frac{1}{R_p R_{p,i}}\sum_{j,k=1}^N a_{ik}\lt\lal e^{i \theta_j}, e^{i \theta_k} \rt\ral\cr
&=\frac{\bar a R_p}{R_{p,i}} N^2 + \frac{1}{R_p R_{p,i}}\sum_{j,k=1}^N \lt(a_{ik} - \bar a\rt)\lt\lal e^{i \theta_j}, e^{i \theta_k} \rt\ral\cr
&\geq \frac{\bar a R_p}{R_{p,i}} N^2 - \frac{\delta N^2}{R_p R_{p,i}}.
\end{aligned}
\end{align}
On the other hand, it follows from \eqref{F-15} and \eqref{F-15-1} that 
\[
\frac{R_p}{R_{p,i}} \geq \frac{1}{C_0} \quad \mbox{and} \quad \frac{1}{R_p R_{p,i}} \leq \frac{C_0}{\bar a(I_0 - 3\delta N) - \delta^2}.
\]
This together with \eqref{est_cos} yields the desired result.
\end{proof}
\begin{remark}Note that if $\delta = 0$, then the right hand side of the inequality \eqref{F-19} becomes $1$, and this yields $\cos(\varphi_{p,i} - \varphi_p) = 1$. 
Indeed, since $\delta = 0$ implies that $a_{ij} \equiv \bar a$, thus we do not need to take into account the local order parameter.
\end{remark}

%
%
%
%
\section{An infinite ensemble of Kuramoto oscillators}\label{sec:6}
\setcounter{equation}{0}
In this section, we discuss the uniform-in-time mean-field limit  from the Kuramoto model for identical oscillators with $m_i = m,~\gamma_i = \gamma,~\nu_i = 0$:
\begin{align}
\begin{aligned} \label{G-0}
\displaystyle {\dot \theta}_i &= \omega_i, \quad t > 0,~~i = 1, \cdots, N, \cr
\displaystyle {\dot \omega}_i &= \frac{1}{m} \lt( -\gamma \omega_i  + \frac{\kappa}{N} \sum_{j=1}^N \sin(\theta_j - \theta_i)  \rt)
\end{aligned}
\end{align}
to the corresponding kinetic Kuramoto equation:
 \begin{align}
 \begin{aligned} \label{G-1}
\displaystyle  \partial_t f + \partial_\theta(\omega f) +\partial_\omega (\mathcal{F}_a[f]f) = 0, \quad (\theta,\omega) \in \T \times \R , \quad t > 0, \cr
\displaystyle \mathcal{F}_a[f](\theta,\omega,t) = \frac1m\lt( -\gamma \omega + \kappa \int_{\T \times \R}\sin(\theta_* - \theta) f(\theta_*, \omega_*,t) \,d\theta_* d \omega_* \rt), 
\end{aligned}
\end{align}
and provide the global existence of measure-valued solutions to \eqref{G-1}.

\subsection{A measure-theoretical minimum} \label{sec:6.1} In this subsection, we briefly discuss the measure-theoretical minimum needed for our discussion. Recall $\mathcal{P}_2(\T \times \R)$ is the set of all Radon measures on $\T \times \R$ with finite second moment and unit mass and Radon measure with unit mass can be understood as a normalized nonnegative bounded linear functional on ${\mathcal C}_0(\T \times \R)$. For a Radon measure $\mu\in\mathcal{P}(\T \times \R)$ with unit mass, we use a standard duality relation:
\[\langle\mu,f\rangle=\int_{\T \times \R}f(\theta, \omega)d\mu(\theta,\omega),\quad f= f(\theta,\omega)\in {\mathcal C}_0(\T \times \R).\]
Below, we recall several definitions that will be used in the next subsection.
\begin{definition} \label{D6.1}
\emph{\cite{H-K-Z}} For $T\in[0,\infty)$, $\mu_t\in \mc_w([0,T);\mathcal{P}_2(\T \times \R)$ is a measure-valued solution to \eqref{G-1} with initial data $\mu^{in} \in\mathcal{P}_2(\T \times \R)$ 
if the following three assertions hold:\newline
\begin{enumerate}
\item Total mass is normalized: $\langle\mu_t,1\rangle=1$.
\item $\mu$ is weakly continuous in $t$:
\[ \langle\mu_t,\phi \rangle~\mbox{is continuous in $t$} \quad \forall~\phi = \phi(\theta, \omega, t)  \in {\mathcal C}_0^1(\T \times \R \times [0,T)). \]
\item $\mu$ satisfies the equation \eqref{G-1} in a weak sense: for $\forall~\varphi \in {\mathcal C}^1_0(\T \times \R \times [0, T))$,
\begin{equation*}
 \langle\mu_t, \varphi(\cdot,\cdot,t)\rangle-\langle \mu_0, \varphi(\cdot,\cdot,0)\rangle=
\int_0^t\langle\mu_s,\partial_s \varphi + \omega \cdot\nabla_x \varphi + F_a(\mu_t) \cdot\nabla_v \varphi \rangle ds,
\end{equation*}
where $F_a(\mu_t)$ is defined as follows.
\[ F_a(\mu_t)(\theta, \omega, t) = \frac1m\lt( -\gamma \omega + \kappa \int_{\T \times \R}\sin(\theta_* - \theta) \mu_t(d\theta_* d \omega_*) \rt)      \]
\end{enumerate}
\end{definition}
\begin{remark} \label{R6.1}
For any solution $\{(\theta_i, \omega_i) \}$ to \eqref{G-0}, the empirical measure 
\[ \mu^N_t := \frac{1}{N} \sum_{i=1}^{N} \delta_{\theta_i} \otimes \delta_{\omega_i}, \]
is a measure-valued solution in the sense of Definition \ref{D6.1}. Thus, ODE solution to \eqref{G-0} can be understood as a measure-valued solution for \eqref{G-1}.  In this way, we can 
treat ODE solution and PDE solution in the same framework. 
\end{remark}
Next, we discuss how to measure the distances between two Radon measures by equipping a metric to the probability measure space ${\mathcal P}_2(\T \times \R)$. We also reivew the concept of local-in-time mean-field limit. To do this, we can consider the Wasserstein $p$-distance $W_p$ in the space ${\mathcal P}(\T \times \R)$.
\begin{definition}\label{D6.2}
\emph{\cite{Ne, Vi2}}
\begin{enumerate}
\item
For $\mu, \nu \in {\mathcal P}_2(\T \times \R)$, Wasserstein $2$-distance $W_{2}(\mu, \nu)$ is defined as follows.
\[
W_{2}(\mu,\nu):= \inf_{\gamma\in\Gamma(\mu,\nu)} \lt( \int_{\T^2 \times \R^2 } \|z-z^*\|^2 d\gamma(z, z^*) \rt)^{\frac{1}{2}},  
\]
where $\Gamma(\mu,\nu)$ denotes the collection of all probability measures on $\T^2 \times \R^2$ with marginals $\mu$ and $\nu$.
\item
For any $T \in (0, \infty)$, the kinetic equation \eqref{G-1} can be derived from the particle model \eqref{G-0} in $[0,T)$, or equivalent to say that the mean-field limit from the particle system \eqref{G-0} 
to the kinetic equation \eqref{G-1}, which is valid in $[0,T)$, if for every solution $\mu_t$ of the kinetic equation \eqref{G-1} with initial data $\mu_0$, the following condition holds: 
\[ \lim_{N\rightarrow +\infty}W_2(\mu_0^N,\mu_0)=0 \quad \Longleftrightarrow \quad  \lim_{N\rightarrow +\infty}W_2(\mu_t^N,\mu_t)=0,\]
where $\mu^N_t$ is a measure valued solution of the particle system \eqref{G-0} with initial data $\mu^N_0$.
\end{enumerate}
\end{definition}

\subsection{Uniform-in-time stability estimate} \label{sec:6.2}  Next, we introduce diameter functions and some other parameters that we will use in the remaining parts of this section. 
\begin{eqnarray*}
&& D(\Theta(t)) := \max_{1\leq i,j \leq N} |\theta_i(t) - \theta_j(t)|, \quad  \Gamma(\theta) := \frac{\sin \theta}{\theta}, \quad \mbox{for some } \theta \in \T, \cr
&& C_k(m,\Theta(t), {\dot \Theta}(t)) := \max\{D(\Theta(t)), D(\Theta(t)) + km\dot{D}(\Theta(t)) \}.
\end{eqnarray*}
Note that the diameter function defined above is well-defined when $D(\Theta) < \pi$. We also note that
\[ D(\Theta) \geq 0 \quad \mbox{and} \quad \dot{D}(\Theta) \leq D({\dot \Theta}). \]

\noindent Now, we show the existence of trapping regions by using a similar argument to the one in \cite[Proposition 4.1]{CHY1}. 
\begin{lemma}\label{L6.1} Given $m$ and $\kappa >0$, let $\Theta = \Theta(t)$ and ${\tilde \Theta} ={\tilde \Theta}(t)$ be two solutions to \eqref{main} 
with initial data $\Theta^{in}$ and ${\tilde \Theta}^{in}$ satisfying
\[
0 \leq C_1(m,\Theta^{in}, {\dot \Theta}^{in}) + C_1(m, {\tilde \Theta}^{in}, {\dot {\tilde \Theta}^{in}}) < \pi.
\]
Then, we have
\[
D({\Theta}(t)) + D({\dot {\tilde \Theta}}(t)) \leq  C_1(m,\Theta^{in}, {\dot \Theta}^{in}) + C_1(m, {\tilde \Theta}^{in}, {\dot {\tilde \Theta}^{in}}), \quad t \geq 0.
\]
\end{lemma}
\begin{proof}Since the proof is very similar to \cite[Proposition 4.1]{CHY1}, 
we only give the main idea here. In \cite[Proposition 4.1]{CHY1}, it is shown
that if $0 \leq C_1(m,\Theta^{in}, {\dot \Theta}^{in}) < \pi$,  then 
\[ D(\Theta(t)) \leq C_1(m,\Theta^{in}, {\dot \Theta}^{in}). \]
Thus, under our assumption, we get $D(\Theta(t)) \leq C_1(m,\Theta^{in}, {\dot \Theta}^{in})$. Similarly, we also obtain 
$D({\tilde \Theta}(t)) \leq C_1(m,\Theta^{in}, {\dot \Theta}^{in})$.
\end{proof}

\vspace{0.5cm}

\noindent Next, we follow the strategies used in the proof of \cite[Theorem 3.4]{CHLXY} and \cite[Theorem 3.1]{CHN}.  We set 
\[ x_i := \theta_i - {\tilde \theta}_i \quad \mbox{and} \quad v_i := \omega_i - {\tilde \omega}_i, \quad i = 1, \cdots, N.  \]
Then $x_i$ and $v_i$ satisfy the following system:
\begin{align}
\begin{aligned} \label{G-2}
\dot{x_i} &= v_i, \quad i=1,\cdots, N, \quad t > 0,\cr
m\dot{v_i} &= -\gamma v_i + \frac{\kappa}{N}  \sum_{j=1}^N\lt(\sin(\theta_j - \theta_i) - \sin(\tilde\theta_j - \tilde\theta_i)\rt).
\end{aligned}
\end{align}
We set 
\[ X := (x_1, \cdots, x_N), \qquad V := (v_1, \cdots, v_N), \]
and define an energy functional $\mathcal{E}_\e$ depending $\e >0$:
\[
\mathcal{E}_\e(X,V) := \e \gamma \|X\|^2 + 2m\e\lal X,V \ral + m\|V\|^2,
\]
where $\lal \cdot, \cdot \ral$ stands for the standard inner product defined on $\R^N$. For notational simplicity, we set
\[
C_k(t):=C_k(m,\Theta(t), {\dot \Theta}(t)) \quad \mbox{and} \quad \tilde C_k(t):=C_k(m,{\tilde \Theta}(t),\dot {\tilde \Theta}(t)).
\]
Note that if $\e \in (0,\gamma/(2m))$, then there exist positive constants $C_0$ and $C_1$  such that
\bq\label{G-2-1}
C_0\|(X,V)\|^2 \leq \me_\e(X,V) \leq C_1\|(X,V)\|^2.
\eq
In the following lemma, we estimate the time derivative of the energy functional $\me_\e$.
\begin{proposition}\label{P6.1} 
Suppose that initial data $(\Theta^{in}, {\tilde \Theta}^{in})$ and parameters $m, \gamma$ and $\kappa$ satisfy the following conditions:
\begin{eqnarray*}
&& \theta_c(t) = \tilde\theta_c(t) = \omega_c(t) = \tilde\omega_c(t) = 0 \quad \mbox{for all $t \geq 0$}, \cr
&& 0 \leq C_1(0) + {\tilde C}_1(0) < \pi, \quad \gamma > \frac{1}{2}, \quad  0<  mK \leq \tilde\Gamma := \frac{\sin (2(C_1(0) + \tilde C_1(0)))}{C_1(0) + \tilde C_1(0)}, 
\end{eqnarray*}
and let $\Theta$ and $\tilde \Theta$ be two solutions to \eqref{G-0}. Then, for some $\e \in [\kappa/(2\tilde\Gamma), \frac{2\gamma - 1}{2m}]$, we have
\[
\frac{d}{dt} \me_\e(X(t),V(t)) + \min\lt\{(2\gamma -1 - 2m\e), \kappa(2\tilde\Gamma \e - \kappa)\rt\}\|(X,V)\|^2 \leq 0 \quad \mbox{for all }~t > 0.
\]
\end{proposition}

\begin{proof} We use the ideas in \cite{CHN, CHLXY}. For this, we divide its proof into three steps. \newline

\noindent $\bullet$~ {\bf Step A}: We take an inner product $2v_i$ with $\eqref{G-1}_2$ and sum it over $i$ to obtain
\begin{align}
\begin{aligned} \label{G-3}
m\frac{d}{dt}\|V\|^2 &= -2 \gamma \|V\|^2 + \frac{2\kappa}{N}\sum_{i,j=1}^N \lt(\sin(\theta_j - \theta_i) - \sin(\tilde\theta_j -\tilde\theta_i))\rt)v_i\cr
&\leq -2 \gamma \|V\|^2 + \frac{\kappa}{N}\sum_{i,j=1}^N \lt(\sin(\theta_j - \theta_i) - \sin(\tilde\theta_j -\tilde\theta_i))\rt)(v_i - v_j)\cr
&\leq -(2\gamma -1)\|V\|^2 + \kappa^2\|X\|^2.
\end{aligned}
\end{align}
Here, we used the relationships
\[
\lt|\lt(\sin(\theta_j - \theta_i) - \sin(\tilde\theta_j -\tilde\theta_i))\rt)(v_i - v_j)\rt| \leq |x_j - x_i||v_i - v_j| \leq \frac{\kappa}{2}|x_j - x_i|^2 + \frac{1}{2\kappa}|v_j - v_i|^2,
\]
\[
\sum_{i,j=1}^N |x_i - x_j|^2 = 2N\|X\|^2, \quad \mbox{and} \quad \sum_{i,j=1}^N |v_i - v_j|^2 = 2N\|V\|^2.
\]
\noindent $\bullet$~{\bf  Step B}:  We take the inner product between $2x_i$ and $\eqref{G-2}_2$ and sum it over all $i$ using $\eqref{G-2}_1$ to obtain
\begin{align*}
\begin{aligned} 
&2m\sum_{i=1}^N \dot{v_i} x_i  \\
& \hspace{0.2cm} = -\gamma \frac{d}{dt}\|X\|^2 + \frac{2K}{N}\sum_{i,j=1}^N\lt(\sin(\theta_j - \theta_i) - \sin(\tilde\theta_j - \tilde\theta_i) \rt)x_i\cr
&  \hspace{0.2cm}  =-\gamma \frac{d}{dt}\|X\|^2 + \frac{K}{N}\sum_{i,j=1}^N\lt(\sin(\theta_j - \theta_i) - \sin(\tilde\theta_j - \tilde\theta_i) \rt)(x_i - x_j)\cr
&  \hspace{0.2cm}  =-\gamma \frac{d}{dt}\|X\|^2 - \frac{4K}{N}\sum_{i,j=1}^N \cos\lt(\frac{\theta_j - \theta_i + (\tilde\theta_j - \tilde\theta_i)}{2}\rt)\sin\lt(\frac{x_j - x_i}{2}\rt)\lt(\frac{x_j - x_i}{2}\rt).
\end{aligned}
\end{align*}
On the other hand, it follows from Lemma \ref{L6.1} that we have
\begin{eqnarray*}
&&  |x_j(t) - x_i(t)| \leq D(\Theta(t)) + D({\tilde\Theta}(t)) \cr
&& \hspace{0.5cm}  \leq C_1(m,\Theta^{in}, {\dot \Theta}^{in}) + C_1(m, {\tilde \Theta}^{in}, {\dot {\tilde \Theta}^{in}}) =: C_1(0) + {\tilde C}_1(0) < \pi \quad \mbox{for all $t \geq 0$.} 
\end{eqnarray*}
This gives
\[
\sin\lt(\frac{x_j - x_i}{2}\rt)\lt(\frac{x_j - x_i}{2}\rt) \geq \Gamma\lt(\frac{C_1(0) + \tilde C_1(0)}{2}\rt)\frac{|x_j - x_i|^2}{4},
\]
and
\[
\cos\lt(\frac{\theta_j - \theta_i + (\tilde\theta_j - \tilde\theta_i)}{2}\rt) \geq \cos\lt(\frac{C_1(0) + \tilde C_1(0)}{2}\rt) > 0.
\]
Since
\[
\Gamma\lt(\frac{C_1(0) + \tilde C_1(0)}{2}\rt)\cos\lt(\frac{C_1(0) + \tilde C_1(0)}{2}\rt) = \frac{\sin (2(C_1(0) + \tilde C_1(0)))}{C_1(0) + \tilde C_1(0)} = \tilde \Gamma > 0,
\]
we have
\bq\label{G-5}
2m\sum_{i=1}^N \dot{v_i} x_i \leq -\gamma \frac{d}{dt}\|X\|^2 -\frac{ \tilde\Gamma \kappa}{N} \sum_{i,j=1}^N |x_j - x_i|^2 = -\gamma \frac{d}{dt}\|X\|^2 - 2\tilde\Gamma \kappa \|X\|^2.
\eq
On the other hand, we get
\[
\dot v_i \cdot x_i = \frac{d (x_i \cdot v_i)}{dt} - |v_i|^2.
\]
This together with \eqref{G-5} yields
\begin{equation} \label{G-6}
\frac{d}{dt}\lt(2m \lal X, V\ral + \gamma \|X\|^2 \rt) + 2\tilde\Gamma \kappa \|X\|^2 \leq 2m\|V\|^2.
\end{equation}

\vspace{0.2cm}

\noindent $\bullet$~ {\bf Step C}:  Multiplying \eqref{G-6} by $\e$ and adding \eqref{G-3}, we get
\[
\frac{d}{dt} \me_\e(X(t),V(t)) + (2\gamma -1 - 2m\e)\|V\|^2 + \kappa (2\tilde\Gamma \e - \kappa)\|X\|^2 \leq 0 \quad \mbox{for all } t > 0.
\]
Then we use the inequalities $2\gamma -1 - 2m\e \geq 0$ and $2\tilde\Gamma \geq \kappa$ to obtain the desired result.
\end{proof}
As a direct corollary of Proposition \ref{P6.1} and the relation \eqref{G-2-1}, we obtain an exponential stability estimate of \eqref{G-1}.  

\begin{corollary}  \label{C6.1} Suppose that the same assumptions in Proposition \ref{P6.1} hold, and let $\e \in (\kappa/(2\tilde\Gamma), (2\gamma-1)/(2m))$. Then, there exists a positive constant $C_3$ such that 
\[ \me_\e(X(t),V(t))  \leq \me_\e(X^{in},V^{in}) e^{-C_3 t}, \quad t \geq 0. \]
\end{corollary}
\begin{proof} It follows from Proposition \ref{P6.1} that we have
\begin{equation} \label{G-6}
 \frac{d}{dt} \me_\e(X(t),V(t)) + C_2 \|(X,V)\|^2 \leq 0 \quad \mbox{for all }~t > 0,
\end{equation} 
 where $c_1$ is a positive constant defined as follows.
\[
C_2:= \min\lt\{(2\gamma -1 - 2m\e), \kappa(2\tilde\Gamma \e - \kappa)\rt\} > 0.
\]
Then, we combine \eqref{G-6} and the relation \eqref{G-2-1} to obtain
\begin{equation*}
 \frac{d}{dt} \me_\e(X(t),V(t)) + \frac{C_2}{C_1} {\mathcal E}_\e(X(t),V(t)) \leq 0, \quad t > 0.
\end{equation*}
This implies the desired decay estimate with $C_3 := C_2/C_1$. 
\end{proof}

\subsection{Global existence of measure-valued solutions} \label{sec:6.3}

We now apply the exponential stability estimate obtained in Proposition \ref{P6.1} to prove Theorem \ref{T3.6} the uniform-in-time mean-field limit. As we mentioned, we use the strategy proposed in \cite{H-K-Z}. We also refer to \cite{BGM} for the stochastic case. 

\begin{proof}[{\bf Proof of Theorem \ref{T3.6}}] It follows from Corollary \ref{C6.1} that
\[
\|X(t) \|^2 + \|V(t))\|^2 \leq \frac{C_1}{C_0} \Big(  \|X^{in} \|^2 + \|V^{in}\|^2 \Big) \quad \mbox{for all } t \geq 0,
\]
i.e., we have the following uniform-in-time stability estimate:
\[
\|\Theta(t) - \tilde\Theta(t) \| + \|\Omega(t) - {\tilde\Omega}(t) \| \leq C_2 \lt( \|\Theta^{in} - {\tilde\Theta}^{in} \| + \| \Omega^{in} - {\tilde\Omega}^{in} \|\rt) \quad \mbox{for all } t \geq 0,
\] 
for some $C_2 > 0$. Then by the direct application of \cite[Corollary 1.1]{H-K-Z}. This completes the proof.
\end{proof}

\subsection{Asymptotic behavior of measure-valued solution} \label{sec:6.4}
\noindent In this subsection, we study the asymptotic formation of phase-locked states of the kinetic Kuramoto mode \eqref{G-1} with positive inertia $m > 0$. For  system \eqref{G-1}, 
the order parameter $R_k$ and the average phase of oscillators $\varphi_k$ are defined by the relationship
\[
R_k(t)e^{{\mathrm i}\varphi_k(t)} :=\int_\T \int_{\R } e^{{\mathrm i}\theta}f(\theta, \omega, t)\, d \omega d\theta, \quad t \geq 0,
\]
where the subscript ``$k$" in $R_k$ and $\varphi_k$ denote the ``{\it kinetic}".    Using these order parameters, the forcing term $\mathcal{F}_a[f]$ can be rewritten as
\[
\mathcal{F}_a[f](\theta,\omega,t) = -\frac1m\lt( \gamma \omega + \kappa R_k \sin(\theta-\varphi_k)\rt).
\]
As discussed before, to get the uniform-in-time mean-field
limit and global-in-time measure-valued solutions to  \eqref{G-1},
we need suitable assumptions on the initial data and parameters.
However, in this section, we assume that there exists a global-in-time measure valued solution $f \in \mc_w(\R_+; \pp_2(\T \times \R))$.  

Consider the characteristic system associated to \eqref{G-1}: for $(\theta,\omega) \in \T \times \R$, 
\begin{equation}\label{G-7}
\begin{cases}
\displaystyle \frac{\pa}{\partial t} q(t;0,\theta,\omega) =p(t;0,\theta,\omega),\\[2mm]
\displaystyle \frac{\pa}{\pa t} p(t;0,\theta,\omega) = \mathcal{F}_a[f](q(t;0,\theta,\omega),p(t;0,\theta,\omega),t), \\[2mm]
\displaystyle \lt( q(0;0,\theta,\omega), p(0;0,\theta,\omega)\rt) = (\theta,\omega).
\end{cases}
\end{equation}
The above characteristic system is well-defined, since the force field is globally Lipschitz. Since $(q(t;0,\cdot,\cdot), p(t;0,\cdot,\cdot)) \# f^{in} = f(t)$, we have
\begin{align*}
\begin{aligned} 
&\mathcal{F}_a[f](q(t;0,\theta,\omega),p(t;0,\theta,\omega),t)\cr
&\hspace{0.2cm} =-\frac1m\lt( \gamma p(t;0,\theta,\omega) + \kappa \int_{\T \times \R}\sin(\theta_* - q(t;0,\theta,\omega)) f(\theta_*, \omega_*,t) \,d \theta_* d \omega_* \rt) \\
&\hspace{0.2cm} =-\frac1m\lt( \gamma p(t;0,\theta,\omega) + \kappa \int_{\T \times \R}\sin(q(t;0,\theta_*,\omega_*) - q(t;0,\theta,\omega)) f^{in}(\theta_*, \omega_*) 
\,d \theta_* d \omega_* \rt).
\end{aligned}
\end{align*}
For notational simplicity, we set
\[
q_t(\theta,\omega) := q(t;0,\theta,\omega) \quad \mbox{and} \quad p_t(\theta,\omega) := p(t;0,\theta,\omega).
\]
We also introduce energy functions:
\begin{align*}
\begin{aligned}
{\bar \me}(t) &:= \frac12 \int_{\T \times \R} p_t^2 (\theta,\omega)\, f^{in}(\theta, \omega) \, d\theta d\omega \cr
&+ \frac{\kappa}{2m}\int_{\T^2 \times \R^2} \lt(1 - \cos(q_t(\theta_*,\omega_*) - q_t(\theta,\omega))\rt)\,f^{in} (\theta_*, \omega_*) f^{in}(\theta, \omega) \,d\theta_* d\omega_* d\theta d\omega \cr
&=: {\bar \me}_K(t) + {\bar \me}_P(t).
\end{aligned}
\end{align*}
Note that the forcing term $\mathcal{F}_a[f]$ can be rewritten as follows:
\begin{equation} \label{G-9}
\mathcal{F}_a[f](q_t(\theta,\omega),p_t(\theta,\omega)) = -\frac1m\lt(\gamma q_t(\theta,\omega) + \kappa R_k(t)\sin(p_t(\theta,\omega) - \varphi_k(t))  \rt).
\end{equation}
\begin{lemma}\label{L6.2} Let $d\mu_t = f(t) d\theta d\omega \in \mc_w(\R_+; \pp_2(\T \times \R))$ be a measure-valued solution to the equation \eqref{G-1}. Then, we have the following assertions.
\begin{itemize}
\item[(i)] (Propagation of averages):
\[
\int_{\T \times \R} q_t(\theta,\omega)\,f^{in}(\theta, \omega) d\theta d\omega = \int_{\T \times \R} \theta \,f^{in}(\theta, \omega)\, d\theta d\omega + m(1 - e^{-\frac{\gamma t}{m}}) \int_{\T \times \R} \omega \,f^{in}(\theta, \omega)\, d\theta d\omega 
\]
and
\[
\int_{\T \times \R} p_t(\theta,\omega)\,f^{in}(\theta, \omega)\, d\theta d\omega = e^{-\frac{\gamma t}{m}} \int_{\T \times \R} \omega \,f^{in}(\theta, \omega)\, d\theta d\omega.
\]
\item[(ii)] (Energy estimate): 
\[
\frac{d{\bar \me}(t)}{dt} + \frac{2\gamma}{m} {\bar \me}_K(t) = 0 \quad \mbox{for} \quad t > 0.
\]
\item[(iii)] (Decay of kinetic energy):
\[
{\bar \me}_K(t)= \frac12 \int_{\T \times \R} p_t^2 (\theta,\omega)\, f^{in}(\theta, \omega)\, d\theta d\omega \to 0 \quad \mbox{as} \quad t \to \infty.
\]
\end{itemize}
\end{lemma}
\begin{proof} The proofs of (ii) and (iii) are very similar to the proofs of Proposition \ref{P4.1} and Corollary \ref{C4.1}.  \newline

\noindent (i) Due to the antisymmetry, we obtain
\[
\frac{d}{dt}\int_{\T \times \R} p_t(\theta,\omega) f^{in}(\theta, \omega)\, d\theta d\omega = -\frac{\gamma}{m} \int_{\T \times \R} p_t(\theta,\omega)f^{in}(\theta, \omega)\, d\theta d\omega.
\]
Additionally, by definition we have
\[
\frac{d}{dt}\int_{\T \times \R} q_t(\theta,\omega)f^{in}(\theta, \omega)\, d\theta d\omega = \int_{\T \times \R} p_t(\theta,\omega)f^{in}(\theta, \omega)\, d\theta d\omega.
\]
Solving the above two differential equations, we obtain
\[
\int_{\T \times \R} q_t(\theta,\omega)\,f^{in} (\theta, \omega)\,d\theta d\omega = \int_{\T \times \R} \theta \,f^{in} (\theta, \omega)\,d\theta d\omega + m (1 - e^{-\frac{\gamma t}{m}}) \int_{\T \times \R} \omega \,f^{in} (\theta, \omega)\,d\theta d\omega,
\]
and
\[
\int_{\T \times \R} p_t(\theta,\omega)\,f^{in} (\theta, \omega)\,d\theta d\omega = e^{-\frac{\gamma t}{m}} \int_{\T \times \R} \omega \,f^{in} (\theta, \omega)\,d\theta d\omega.
\]

\vspace{0.2cm}

\noindent (ii) By direct calculation, we have
\begin{align*}
\begin{aligned}
&\frac12\frac{d}{dt}\int_{\T \times \R} |p_t(\theta,\omega)|^2  f^{in} (\theta, \omega) \,d\theta d\omega \\
& \hspace{1cm}  = \int_{\T \times \R} p_t(\theta,\omega) \cdot \pa_t p_t(\theta,\omega)\,f^{in} (\theta, \omega) \,d\theta d\omega \\
&  \hspace{1cm}  = - \frac{\gamma}{m} \int_{\T \times \R} p_t^2 (\theta,\omega)\, f^{in}(\theta, \omega) \, d\theta d\omega  \\
&  \hspace{1cm}   +\frac{\kappa}{m}\int_{\T^2 \times \R^2} \sin(q_t(\theta_*,\omega_*) - q_t(\theta,\omega)) p_t(\theta,\omega)\,f^{in}(\theta_*, \omega_*) 
f^{in}(\theta, \omega)  \,d\theta_* d\theta d\omega_* d\omega \\
&  \hspace{1cm}  = - \frac{\gamma}{m} \int_{\T \times \R} p_t^2 (\theta,\omega)\, f^{in}(\theta, \omega) \,d\theta d\omega \\
&  \hspace{1cm}  +\frac{\kappa}{2m}\frac{d}{dt}\int_{\T^2 \times \R^2} \cos(q_t(\theta_*,\omega_*) - q_t(\theta,\omega))\,f^{in}(\theta_*, \omega_*) f^{in}(\theta, \omega)\, d\theta_* d\theta 
d\omega_* d\omega.
\end{aligned}
\end{align*}
Thus, we obtain
\[
\frac{d}{dt} {\bar \me}(t) + \frac{2\gamma}{m} {\bar \me}_K(t) = 0.
\]
Here, we used the fact that $\int_{\T \times \R} f^{in}(\theta, \omega) \,d\theta d\omega = 1$. 

\vspace{0.2cm}

\noindent (iii) To obtain the desired result, it is enough to show that $\lt|\dot {\bar \me}_K (t)\rt| \leq C$ for some $C > 0$ for all $t \geq 0$. In fact, proceeding as in Corollary \ref{C4.1}, we have
\[
\lt|\frac{d}{dt} {\bar \me}_K(t) \rt| \leq \frac2m {\bar \me}(0) + \frac{\kappa \sqrt2}{m}\sqrt{{\bar \me}(0)} < 0 \quad \mbox{uniformly in }t.
\]
This together with the integrability of $ {\bar \me}_K $ concludes the desired decay estimate.
\end{proof}
\begin{remark}\label{R6.2}
\noindent 1. Since the interaction term is bounded by the coupling strength $\kappa$, it follows from \eqref{G-7} that we have 
\begin{equation} \label{G-6}
|p_t(\theta,\omega)| \leq |\omega| e^{-\frac{\gamma t}{m}} + m \kappa(1 - e^{\frac{\gamma t}{m}}) \leq \max\{ |\omega|, m\kappa\} \quad \mbox{for every} \quad (\theta,\omega) \in \mbox{supp}(f^{in}).
\end{equation}
This implies that if $f^{in}$ is compactly supported in $\omega$, then support of $f(t)$ in $\omega$ uniformly bounded in $t$. \newline

\noindent 2. Since $(q_t(\cdot,\cdot),p_t(\cdot,\cdot))\# f^{in} = f(t)$, Lemma \ref{L6.2} and the estimate \eqref{G-6} imply that if $f^{in}$ is compactly supported in $\omega$, then we obtain
\begin{align*}
\begin{aligned}
&\frac12 \frac{d}{dt}\lt(\int_{\T \times \R} w^2 f (\theta, \omega)\, d\theta d\omega + \frac{\kappa}{2}\int_{\T^2 \times \R^2} \lt(1 - \cos(\theta_* - \theta)\rt)f(\theta_*, \omega_*,t) f(\theta, \omega, t) 
d\theta_* d\theta d\omega_* d\omega \rt)\cr
&\vspace{0.5cm} = - \frac{\gamma}{m} \int_{\T \times \R} w^2\, f(\theta, \omega)\, d\theta d\omega,
\end{aligned}
\end{align*}
and
\[
\int_{\T \times \R} w^2\, f (\theta, \omega)\,d\theta d\omega \to 0 \quad \mbox{as} \quad t \to \infty.
\]
\newline

\noindent 3. Without loss of generality, we set
\[
\int_{\T \times \R} \theta \,f_0(d\theta,d\omega) = \int_{\T \times \R} \omega \,f_0(d\theta,d\omega) =0.
\]
\end{remark}
\noindent We now present a characterization of  stationary solutions $f^*$ for the equation \eqref{G-1} which corresponds to the continuum version of the content in Section \ref{sec:5.4}.

\begin{proposition} \label{P6.2}
\cite{BCM, H-K-R}
The function $f_e = f_e(\theta,\omega)$ is a stationary solution to \eqref{G-1} if and only if one of the following identities holds:
\begin{itemize}
\item[(i)] $R_k \equiv 0$.
\item[(ii)] $f_e(\theta,\omega)$ is of type $(c_1,c_2)$, which means that $f_e(\theta,\omega) = \lt(c_1\delta_{\varphi^*}(\theta) + c_2\delta_{\varphi^*+\pi}(\theta)\rt)\otimes \delta_0(\omega)$, where $c_1 > c_2 \geq 0$ satisfying $c_1 + c_2 = 1$,
\end{itemize}

\end{proposition}
\begin{proof}
We refer to \cite{BCM, H-K-R} for a detailed proof. 
\end{proof}

We are now ready to provide a proof for Theorem \ref{T3.7}.  \newline
\begin{proof}[{\bf Proof of Theorem \ref{T3.7}}]
We proceed as in Proposition \ref{P4.2}. Note that our goal is to show that
\[
\int_{\T \times \R} |\gamma p_t(\theta,\omega) + \kappa R_k \sin(q_t(\theta,\omega) - \varphi_k)|^2 f^{in} (\theta, \omega)\,d\theta d\omega \to 0 \quad \mbox{as} \quad t \to \infty.
\]
This and the decay of kinetic energy yield
\begin{equation} \label{G-11}
\int_{\T \times \R} R_k^2\sin^2(q_t(\theta,\omega) - \varphi_k)f^{in}(\theta,\omega)\, d\theta d\omega \to 0 \quad \mbox{as} \quad t \to \infty.
\end{equation}
Next, we consider two cases. \newline

\noindent $\bullet$ Case A ($R_k(0) = 0$): In this case, we get 
\[
\int_{\T \times \R} \omega^2 f^{in}(\theta,\omega)\,d\theta d\omega = 0.
\]
Here we used assumption \eqref{su_lt}. This implies that the support of $f^{in}$ in $\omega$ is measure zero. Since $p_t$ is absolutely continuous, this yields the support of $f(t)$ in $\omega$ is measure zero. Thus the support of $f(t)$, in both $\theta$ and $\omega$, is measure zero, and this concludes that $R_k(t) = 0$ for all $t \geq 0$.
\vspace{0.2cm}

\noindent $\bullet$ Case B ($R_k(0) > 0$): By the almost same argument in Section \ref{sec:4.3}, just replacing $\me_K$ and $R_p$ with $\bar \me_K$ and $R_k$, respectively, there exists a $R_* >0$ such that 
\[ R_k (t) \geq R_* > 0. \]
Then, this and \eqref{G-11} yield
\bq\label{G-12}
\int_{\T \times \R} \sin^2(q_t(\theta,\omega) - \varphi_k)f^{in}(\theta,\omega)\,  d\theta d\omega \to 0 \quad \mbox{as} \quad t \to \infty.
\eq

Note that the above observation was enough to show the existence of $\displaystyle R^*:=\lim_{t \to \infty} R_k(t)$ at the particle level, as we discussed in the proof of Theorem \ref{T3.4}. 
However, at the continuum level, we need more information. In fact, we require the integrability of \eqref{G-12} with respect to $t$. For this, we estimate
\begin{align*}
\begin{aligned} 
&\frac12\frac{d}{dt}\int_{\T \times \R} |\gamma p_t(\theta,\omega) + \kappa R_k \sin(\Theta_t(\theta,\omega) - \bar\varphi)|^2 f^{in}(\theta,\omega)\,  d\theta d\omega\cr
& \hspace{1cm} = \int_{\T \times \R} \lt( \gamma p_t(\theta,\omega) + \kappa R_k \sin(\Theta_t(\theta,\omega) - \bar\varphi)\rt)\cr
& \hspace{2cm} \times \pa_t \lt(\gamma p_t(\theta,\omega) + \kappa R_k \sin(q_t(\theta,\omega) - \varphi_k)\rt) f^{in} (\theta,\omega)\, d\theta d\omega \cr
& \hspace{1cm}= -\frac1m\int_{\T \times \R}|\gamma p_t(\theta,\omega) + \kappa R_k \sin(q_t(\theta,\omega) - \varphi_k)|^2f^{in} (\theta,\omega)\, d\theta d\omega\cr
& \hspace{1cm} - \kappa \int_{\T^2 \times \R^2}\lt (\gamma p_t(\theta,\omega) + \kappa R_k \sin(q_t(\theta,\omega) - \varphi_k)\rt) \cos(q_t(\theta,\omega) - q_t(\theta_*,\omega_*)) \cr
& \hspace{2cm} \times \lt(p_t(\theta,\omega) - p_t(\theta_*,\omega_*)\rt)f^{in} (\theta,\omega)f^{in} (\theta_*,\omega_*) \,d\theta d\omega d\theta_*d\omega_*,\cr
& \hspace{1cm}  =: -\frac1m\int_{\T \times \R}|\gamma p_t(\theta,\omega) + \kappa R_k \sin(q_t(\theta,\omega) - \varphi_k)|^2f^{in} (\theta,\omega)\, d\theta d\omega  + 
 {\mathcal I}_{3},
\end{aligned}
\end{align*}
due to \eqref{G-9}. Now, the term ${\mathcal I}_3$ can be further estimated as follows.
\begin{align*}
\begin{aligned} 
{\mathcal I}_3 &\leq \kappa \int_{\T^2 \times \R^2} |\gamma p_t(\theta,\omega) + \kappa R_k \sin(q_t(\theta,\omega) - \varphi_k)|\cr
& \qquad \times |p_t(\theta,\omega) - p_t(\theta_*,\omega_*)|f^{in}(\theta, \omega) f^{in}(\theta_*, \omega_*) \,d\theta_* d\theta d\omega_* d\omega \cr
&\leq \kappa \lt(\int_{\T \times \R}| \gamma p_t(\theta,\omega) + \kappa R_k \sin(q_t(\theta,\omega) - \varphi_k)|^2 f^{in}(\theta, \omega) \,d\theta d\omega \rt)^{1/2} \cr
& \qquad \times \lt(\int_{\T^2 \times \R^2 }|p_t(\theta,\omega) - p_t(\theta_*,\omega_*)|^2f^{in}(\theta, \omega) f^{in}(\theta_*,\omega_*) \,d\theta_* d\theta d\omega_* d\omega \rt)^{1/2}\cr
&\leq 2\kappa \lt(\int_{\T \times \R}|\gamma p_t(\theta,\omega) + \kappa R_k \sin(q_t(\theta,\omega) - \varphi_k)|^2 f^{in}(\theta, \omega) \,d\theta d\omega \rt)^{1/2}\cr
& \qquad \times \lt(\int_{\T \times \R}|p_t(\theta,\omega)|^2 f^{in}(\theta, \omega)\, d\theta d\omega \rt)^{1/2}\cr
&\leq 2\kappa \e\int_{\T \times \R}| \gamma q_t(\theta,\omega) + \kappa R_k \sin(q_t(\theta,\omega) - \varphi_k)|^2 f^{in}(\theta, \omega) \,d\theta d\omega \cr
& \quad + \frac{\kappa}{2\e}\int_{\T \times \R}|p_t(\theta,\omega)|^2 f^{in}(\theta, \omega)\, d\theta d\omega,
\end{aligned}
\end{align*}
where $\e > 0$ will be determined later. Thus, by combining the above estimates, we obtain
\begin{align*}
\begin{aligned}
&\frac12\frac{d}{dt}\int_{\T \times \R} |\gamma p_t(\theta,\omega) + \kappa R_k \sin(q_t(\theta,\omega) - \varphi_k)|^2 f^{in} (\theta, \omega)\,d\theta d\omega \cr
& \hspace{1cm} \leq - \frac1m\lt(\gamma - 2\kappa m\e \rt)\int_{\T \times \R} |p_t(\theta,\omega) + \kappa R_k \sin(q_t(\theta,\omega) - \varphi_k)|^2 f^{in} (\theta, \omega)\,d\theta d\omega \cr
& \hspace{1.5cm} + \frac{\kappa}{2\e}\int_{\T \times \R}|p_t(\theta,\omega)|^2 f^{in} (\theta, \omega)\,d\theta d\omega.
\end{aligned}
\end{align*}
Then, we choose $\e>0$ such that $\gamma - 2Km\e > 0$ and use the integrality of $\bar\me_K$ together with the fact $R(t) \geq R_*>0$ to find
\[
\int_0^\infty \int_{\T \times \R} \sin^2(\Theta_t(\theta,\omega) - \bar\varphi) f^{in}(\theta, \omega)\, d\theta d\omega \,dt < \infty.
\]
Note that
\[
\dot{R}_k(t) = -\int_{\T \times \R} \sin(q_t(\theta,\omega) - \varphi_k(t)) p_t(\theta,\omega)f^{in} (\theta, \omega)\,d\theta d\omega.
\]
Then, by the above estimate, we have
\begin{align*}
\begin{aligned}
\int_0^\infty |\dot{R}_k(t)|\,dt &\leq \lt(\int_0^\infty \int_{\T \times \R} \sin^2(q_t(\theta,\omega) - \varphi_k(t))f^{in} (\theta, \omega,t)\,d\theta d\omega \,dt\rt)^{1/2}\cr
&\times \lt(\int_0^\infty \int_{\T \times \R} |p_t(\theta,\omega)|^2f^{in} (\theta, \omega,t)\,d\theta d\omega \,dt\rt)^{1/2}.
\end{aligned}
\end{align*}
This provides the following relations:
\[
0 < R^* = \lim_{t \to \infty}R_k(t) = R_k(0) + \int_0^\infty \dot{R}_k(t)\,dt,
\]
Consequently, due to Proposition \ref{P6.2},  $f$ converges to the bi-polar state. Thus, we verified the third assertion in Theorem \ref{T3.7}.
\end{proof}

%
%

\section{Conclusion} \label{sec:7}
\setcounter{equation}{0}
In this paper, we addressed the complete synchronization problem for the particle and kinetic Kuramoto models for identical oscillators with homogeneous and heterogenous inertia and friction. The Kuramoto model has been introduced as a mathematical model for the synchronization of the weakly coupled limit-cycle oscillators. Thus, in order to determine under what conditions on the system's parameters and initial conditions, whether does the models exhibit complete synchronization or not is one of fundamental issues from the modeling view point. In mathematical modeling, the effect of inertia is often studied by adding the second order term and when the inertial effect is negligible, it  
is often ignored simply assuming that the singular term containing the inertia is small. For the Kuramoto model, it is easy to see that the effect of inertia does not affect the structure of emergent phase-locked states. 
Thus, the effect of inertia can be seen only in the intermediate temporal regime. In this work,  we have presented three main results. 
First, we extended an admissible set of initial data leading to the complete synchronization to a larger generic set in a large coupling strength regime. Second, we provided a sufficient framework for complete synchronization 
under the effect of heterogeneous inertia and frictions. In some sense, our proposed framework is robust in inertia and friction variations. Finally, we considered the corresponding kinetic model which approximates the large ensemble of identical Kuramoto oscillators. In particular, we provided the global existence of measure-valued solutions and their asymptotic behaviors by lifting corresponding particle results via the rigorous
 uniform-in-time mean-field limit. All presented results deal with only the ensemble of identical oscillators. Thus, how to extend our results to the ensemble of non-identical oscillators will be a next interesting future work to be explored. 

%
%
%
%

\end{document}